\def\MODE{1}

\documentclass[11pt]{article}

\usepackage[colorlinks=true,linkcolor=blue,urlcolor=black,citecolor=blue,breaklinks]{hyperref}
\usepackage{fullpage}
\usepackage[T1]{fontenc}
\usepackage{math}
\usepackage{color}
\usepackage{enumerate}
\usepackage{graphicx}
\usepackage[font={small}]{caption}

\numberwithin{equation}{section}

\begin{document}

\title{
Non-Asymptotic Analysis of Robust Control from \\Coarse-Grained Identification
}

\author{Stephen Tu, Ross Boczar, Andrew Packard, Benjamin Recht%
\if\MODE1
\else
\thanks{S.~Tu, R.~Boczar, A.~Packard, and B.~Recht are with the University of California, Berkeley, CA~94720, USA. \texttt{\{stephent,boczar,apackard,brecht\}@berkeley.edu}}
\fi}
\maketitle


\begin{abstract}
This work explores the trade-off between the number of samples required to accurately
build models of dynamical systems and the degradation of performance in various
control objectives due to a coarse approximation.  In particular, we show that
simple models can be easily fit from input/output data and are sufficient for
achieving various control objectives.
We derive bounds on the number of noisy input/output samples from a stable linear time-invariant system that are sufficient to guarantee that the corresponding finite impulse response approximation is close to the true system in the $\mathcal{H}_\infty$-norm. We demonstrate that these demands are lower than those derived in prior art which aimed to accurately identify dynamical models.  We also explore how different physical input constraints, such as power constraints, affect the sample complexity.
Finally, we show how our analysis fits within the established framework
of robust control, by demonstrating how a controller designed for an
approximate system provably meets performance objectives on the true
system.
\end{abstract}


\section{Introduction}

Most control design relies on establishing a model of the system to be
controlled.  For simple physical systems, a model with reasonable fidelity can
typically be constructed from knowledge of the physics at hand.  However, for
complex, uncertain systems, building models from first principles becomes
quickly intractable and one usually resorts to fitting models from empirical
input/output data.  This approach naturally raises an important question: how well
must we identify a system in order to control it?

In this work, we attempt to answer this question by striking a balance
between system identification and robust control.  We aim to identify coarse
estimates of the true underlying model while coupling our estimation with precise
probabilistic bounds on the inaccuracy of our estimates.  With a coarse
model in hand, we can use standard robust synthesis tools that take into
account the derived bounds on the model uncertainty.

More precisely, given an unknown stable discrete-time plant $G$, we bound the
error accrued by fitting a finite impulse response (FIR) approximation to $G$
from noisy output measurements.  These bounds balance the sample complexity of
estimating an unknown FIR filter against the capability of such a filter to
approximate the behavior of $G$.  In particular, we show that notably short FIR
filters provide a sufficient approximation to stable systems in order to ensure
robust performance for a variety of control design tasks.  In particular, we
demonstrate considerable savings in experimental measurements as compared to
other non-asymptotic schemes that aim to precisely identify $G$.

In the process of fitting a FIR filter, a natural question arises as to
what inputs should be used to excite the unknown system. Of course,
due to actuator limitations and other physical constraints, we are not free to choose
any arbitrary input.
Hence, we model the choice of inputs as an experiment design question,
where the practitioner specifies a bounded input set and asks for the best
$m$ inputs to use to minimize FIR identification error.
We propose a new optimal experiment design procedure for solving this problem,
and relate it to the well studied $A$-optimal experiment design objective from
the statistics literature~\cite{pukelsheim93}.
This connection is used to study practical cases of input constraints.
Specifically, we prove that when the inputs are $\ell_2$-power constrained,
then impulse responses are the optimal choice of inputs.  However, we show that
is not the case when the inputs are $\ell_\infty$-constrained. For
$\ell_\infty$ constraints, we construct a deterministic set of inputs which is
within a factor of 2 to the optimal solution.  Combining these designs with our
probabilistic bounds, we show that for estimating a length-$r$ FIR filter
$G_r$, as long as $m \geq 4r$, the residual $\mathcal{H}_\infty$ error $\norm{G_r -
\widehat{G}_r}_\infty$ on the estimate $\widehat{G}_r$ satisfies
$\widetilde{O}(1/\sqrt{m})$\footnote{The notation $\widetilde{O}(\cdot)$ suppresses dependence on polylogarithmic factors.} with high probability.  This is a
substantial improvement over the $\widetilde{O}(\sqrt{r/m})$ scaling
which we show occurs in the $\ell_2$-constrained case.
We also prove an information-theoretic lower bound which shows that
when the true system happens to be an FIR filter,
our bounds are minimax optimal up to constant factors
for the given estimation problem.

Experimentally, we show that $\Hinf$ loop-shaping controller design on the
estimated FIR model, using probabilistic bounds, can be used to synthesize
controllers with both stability and performance guarantees on the closed loop
with the true plant.  We also demonstrate that our probabilistic bounds can be
estimated directly from data using Monte--Carlo techniques.
We hope that our results encourage further investigation into a rigorous
foundation for data-driven controller synthesis.

\subsection{A sample complexity bound for FIR identification}
\label{sec:intro:results}

We now state our main results: upper and lower bounds on the sample complexity
of FIR system identification.  Let $G$ be a stable, discrete-time SISO
LTI system.  Suppose we are given query access to $G$ via independent, noisy
measurements of the form
\begin{align}\label{eq:query-model}
    Y_{u,T} := (g \ast u)_{k=0}^{T-1} + \xi \:, \:\: \xi \sim N(0, \sigma^2 I_{T}) \:.
\end{align}
Above, $g$ denotes the impulse response of $G$ and
$T$ is the length of the output we observe.  We assume that we are allowed to
choose any input $u$ contained within a bounded set $\mathcal{U}$, which is
specified beforehand.  From these measurements, we can approximate $G$ by a length-$r$ FIR filter
$\widehat{G}_r(z)$ as
\begin{align}\label{eq:g-tilde-r-def}
    \widehat{G}_r(z) := \sum_{k=0}^{r-1} \widehat{g}_k z^{-k} \:,
\end{align}
where $r \leq T$, and the coefficients $\widehat{g}_k$ are estimated from ordinary least-squares
(c.f. Section~\ref{sec:fir_sysid}).
We note that the extra degree of freedom in allowing $r \neq T$
reduces the variance of the higher lag terms, which
is a standard trick used in the system-identification literature
(see e.g. Section 2 of Wahlberg et al.~\cite{wahlberg2010non}).

The main quantity of interest in this setting is the number of timesteps needed
in order to ensure that the $\Hinf$-norm of the error $G -
\widehat{G}_r$ satisfies the bound $\norm{G - \widehat{G}_r}_\infty \leq
\varepsilon$ with probability at least $1-\delta$ over the randomness of the
noise.  Here, the total number of timesteps is the product of the number of
queries of the form~\eqref{eq:query-model} times length of the queries.  That
is, the number of timesteps is $m \times T$, where $m$ is number of
measurements taken and $T$ denotes the length of each measurement.
This quantity depends on the set $\mathcal{U}$; we
restrict (for now) to the case where $\mathcal{U}$ is either the unit $\ell_2$-ball or the unit $\ell_\infty$-ball.  These two sets comprise the most common input constraints
found in the controls literature.  Nevertheless, our analysis later
will cover all $\ell_p$-balls for $p \in [1,\infty]$.

In both cases, we must first consider the length of the FIR
filter~\eqref{eq:g-tilde-r-def} required to ensure reasonable approximation in
the $\mathcal{H}_\infty$-norm.  We must guarantee that we are able to accurately capture
the large components of the impulse response of $G$.  Therefore, we will need
some measure of how quickly the impulse response coefficients tend to zero.  It
turns out that the $\Hinf$-norm of $G$ provides a
convenient proxy for the decay of the impulse response.
Throughout, we will use the following sufficient
bound on the truncation length:
\begin{defn}[Sufficient Length Condition]
Let $G$ be stable with stability radius $\rho \in (0, 1)$.
Fix a $\varepsilon > 0$. Let $R(\varepsilon)$ be the smallest
integer which satisfies
\begin{align}
  R(\varepsilon) \geq \inf_{\rho < \gamma < 1} \: \frac{1}{1-\gamma} \log\left( \frac{\hinfnorm{G(\gamma z)}}{\varepsilon (1-\gamma)} \right) \:. \label{eq:r_condition}
\end{align}
\end{defn}
Equation \eqref{eq:r_condition} characterizes the approximation error of an FIR
filter to $G$ as a balance between the growth of $1/(1-\gamma)$ versus the
decay of the logarithm of $\hinfnorm{G(\gamma z)}$, as $\gamma$ varies between
$(\rho, 1)$.

We first study the $\ell_2$-ball case.  In this case, we will set all $m$
inputs to an impulse; that is, $u_i = e_1$, where $e_1 \in \R^r$ is the first
standard basis vector.
\begin{thm}[Main result, $\ell_2$-constrained case]
\label{thm:main_result_ltwo}
Fix an $\varepsilon > 0$ and $\delta \in (0, 1)$, and suppose that $\mathcal{U}
= \{ x \in \R^T : \norm{x}_2 \leq 1 \}$. Let $G$ be stable with stability
radius $\rho \in (0, 1)$, and set $r \geq R(\varepsilon/2)$ from \eqref{eq:r_condition}
and $T = 2r$.
Set $m$ measurements $u_1, ..., u_m \in \mathcal{U}$, with $u_i = e_1$ for
$i=1, ..., m$, where $m \geq 1$ satisfies
\begin{align}
    m \geq C \frac{\sigma^2 r}{\varepsilon^2} \left(\log{r} + \log\left(\frac{1}{\delta}\right) \right) \:.
\end{align}
Then, with probability at least $1-\delta$,
we have $\norm{G - \widehat{G}_r}_\infty \leq \varepsilon$.
Above, $C$ is an absolute positive constant.
\end{thm}
Theorem~\ref{thm:main_result_ltwo} states that the number of timesteps to
achieve identification error $\varepsilon$ with $\ell_2$-constrained inputs
scales as $\widetilde{O}(\sigma^2 r^2/\varepsilon^2)$
in the regime when $\sigma/\varepsilon \gg 1$.
It also turns out that this input
ensemble is optimal for the $\ell_2$-ball case, which we will discuss shortly.

We next turn to the $\ell_\infty$-ball case.  In this case, we take $m = 2n$ to be an
even number, and construct the measurement ensemble
\begin{align}\label{eq:sinusoid_inputs}
	u^{(\mathrm{c})}_{i,t} = \cos\left(\frac{2\pi i t}{n}\right)~~\quad\mbox{and}\quad~~
	u^{(\mathrm{s})}_{i,t} = \sin\left(\frac{2\pi i t}{n}\right)~~\quad\mbox{for } i=0,..., n-1\,.
\end{align}
With this measurement ensemble, we prove the following result
for $\ell_\infty$-constraints.
\begin{thm}[Main result, $\ell_\infty$-constrained case]
\label{thm:main_result_linfty}
Fix an $\varepsilon > 0$ and $\delta \in (0, 1)$, and
suppose that $\mathcal{U} = \{ x \in \R^T : \norm{x}_\infty \leq 1 \}$.
Let $G$ be stable with stability radius $\rho \in (0, 1)$,
and set $r \geq R(\varepsilon/2)$ from \eqref{eq:r_condition} and $T = 2r$.
Set $m$ measurements as described in
\eqref{eq:sinusoid_inputs}, where $m \geq 4r$ satisfies
\begin{align}
    m \geq C \frac{\sigma^2}{\varepsilon^2} \left(\log{r} + \log\left(\frac{1}{\delta}\right) \right) \:.
\end{align}
Then, with probability at least $1-\delta$,
we have $\norm{G - \widehat{G}_r}_\infty \leq \varepsilon$.
Above, $C$ is an absolute positive constant.
\end{thm}
In the regime when $\sigma/\varepsilon \gg 1$,
Theorem~\ref{thm:main_result_linfty} states that the number of timesteps to
achieve identification error $\varepsilon$ with $\ell_\infty$-constrained
inputs scales as $\widetilde{O}(\sigma^2 r/\varepsilon^2)$. This is substantially
more efficient than the complexity $\widetilde{O}(\sigma^2 r^2/\varepsilon^2)$ which
arises in the $\ell_2$-constrained case.
We conclude by noting that this particular input ensemble is
optimal for the $\ell_\infty$-ball case up to constants,
which we turn our attention to now.

For the lower bound, we assume that $G$ itself is an length-$r$ FIR filter.
We consider the general case when all $m$ inputs
are constrained to a unit $\ell_p$-ball, where $p \in [1, \infty]$.
\begin{thm}[Main result, minimax risk lower bound]
\label{thm:main_result_lower_bound}
Fix a $p \in [1, \infty]$ and $r \geq 16$.
Suppose that $m \geq 1$ measurements $u_1, ..., u_m \in \R^T$ are fixed beforehand,
with $T = 2r$ and $\norm{u_i}_p \leq 1$ for all $i = 1, ..., m$.
Let $\mathscr{H}_r$ denote the space of all length-$r$ FIR filters.
We have that
\begin{align}
  \inf_{\widehat{G}} \sup_{G \in \mathscr{H}_r} \E\hinfnorm{ \widehat{G} - G } \geq C \sigma \sqrt{\frac{r^{2/\max(p, 2)} \log{r}}{m}} \:,
\end{align}
where the infimum ranges over all measurable functions $\widehat{G} : \otimes_{k=1}^{m} \R^T \longrightarrow \mathscr{H}_r$, and $C$ is an absolute positive constant.
\end{thm}
In view of Theorem~\ref{thm:main_result_lower_bound}, we see that
the rates prescribed by Theorem~\ref{thm:main_result_ltwo} for the $\ell_2$-constrained case
and by Theorem~\ref{thm:main_result_linfty} for the $\ell_\infty$-constrained case
are minimax optimal up to constant factors.
We conclude by noting that our choice of $T = 2r$ is arbitrary.
Indeed, the same results hold for $T = \lceil (1+\varepsilon) r \rceil$
for any fixed $\varepsilon > 0$, which only a change in the constant
factors.


\section{Related Work}

\subsection{Transfer function identification}

Estimating the transfer function of a linear time-invariant system from
input/output pairs has been studied in various forms in both the
controls literature~\cite{ljung92,ljung85} and the statistics literature~\cite{gerencser92,goldenshluger01,shibata80}, 
where it is closely related to estimating the coefficients of a stable
autoregressive (AR) process.
The main difference between our work and that of autoregressive estimation
is that we assume the noise process driving the system is chosen
by the practitioner (which we denote as the input to the system), 
and the stochastic component enters only during the output of the
system.
This simplifying assumption allows us to provide stronger non-asymptotic guarantees.
Also by making prior assumptions on the stability radius of the underlying
system, we circumvent the delicate issue of model order selection;
a similar assumption is made in \cite{goldenshluger01}.

Most closely related to our work is that of
Goldenshluger~\cite{goldenshluger98}, where he considers the problem of estimating
the impulse response coefficients of a stable SISO LTI system.
Goldenshluger provides upper and lower bounds on the $\ell_p$-error when the
residual between the estimate and the true coefficients is treated as
a sequence in $\ell_p$ for $p \in [1,\infty]$.
The main difference between Goldenshluger's setting and ours is that 
he restricts himself to the case when the input $u$ is $\ell_\infty$-constrained, 
and furthermore assumes only a single realization is available.
On the other hand, we make assumption that multiple independent
realizations of the system are available, which is reasonable in a controlled
laboratory setting. This assumption simplifies the analysis
and allows us to study more general $\ell_p$-constrained inputs.

\subsection{System identification}

We now turn our attention to system identification, where the classical
results can be found in \cite{ljung99}.
Sample complexity guarantees in the system identification literature often
require strong assumptions, which are difficult to verify. 
Most analyses are asymptotic and are based on the idea of \emph{persistence of
excitation} or \emph{mixing}
\cite{mcdonald2012nonparametric,vidyasagar2008learning}.  There has been some
progress in estimating the sample complexity of dynamical system identification
using machine learning tools~\cite{campi2002finite,vidyasagar2008learning}, but
such results typically yield pessimistic sample complexity bounds that are
exponential in the degree of the linear system or other relevant quantities.  

Two recent results provide polynomial sample complexity for identifying linear
dynamical systems. Shah et al.~\cite{shah2012} show that if certain frequency
domain measurements are obtained from a linear dynamical system, then the
system can be approximately identified by solving a second-order cone
programming problem. The degree of the estimated IIR system scales as
$(1-\rho(A))^{-2}$ where $\rho(A)$ denotes the stability radius.  Similarly,
Hardt et al.~\cite{hardt2016gradient} show that one can estimate an IIR system
from time domain observations with a number of measurements polynomial in
$(1-\rho(A))^{-2}$, under the assumption that the impulse response 
coefficients $\{g_k\}_{k \geq 0}$ satisfy the decay law $\abs{g_k} \leq C \rho(A)^k$,
where $C$ is considered a constant independent of the degree of the system.
In this work, we show that under the same decay assumption, a considerably
smaller FIR approximation with degree $\widetilde{O}((1-\rho(A))^{-1})$
suffices to complete many control design tasks. 

\subsection{Robust control}

Classical robust control literature focuses much of its effort on designing a
controller while taking into account fixed bounds on the uncertainty in the model.
There are numerous algorithms for controller synthesis under various
uncertainty specifications, such as coprime factor
uncertainty~\cite{mcfarlane92} or state-space uncertainty~\cite{packard93}.
However, there are only a few branches of the robust control literature that
couple identification to control design, and the identification procedure best
suited for a particular control synthesis scheme is usually not specified.

\subsection{\texorpdfstring{$\mathcal{H}_\infty$}{H-infinity} identification and gain estimation}

Most related to our work is the literature on $\mathcal{H}_\infty$
identification.  In this literature, noisy input/output data from an unknown
stable linear time-invariant (LTI) plant is collected in either the frequency
or time domain; the goal is often to estimate a model with low
$\mathcal{H}_\infty$ error.  For frequency domain algorithms, see
e.g.~\cite{helmicki91,hindi02}, and for time domain algorithms, see e.g.~\cite{chen93}.
A comprehensive review of this line of work is given by Chen and
Gu~\cite{chen00}.

The main difference between the $\mathcal{H}_\infty$ identification literature and our
work is that we assume a probabilistic noise model instead of worst-case
(adversarial), and we assume that our identification algorithm is allowed to
pick its inputs to the plant $G$.  As we will see, these simplifying
assumptions lead to simple algorithms, straightforward analysis, and
finite-time sample complexity guarantees.

Another related line of work is the use of the power
method~\cite{rojas2012analyzing,wahlberg2010non} for estimating the
$\mathcal{H}_\infty$-norm of an unknown SISO plant.  The key insight 
is that in the SISO case, a time-reversal trick can be applied to effectively
query the system $G^*\circ G$, where $G^*$ denotes the adjoint system. 
This approach is appealing, since the power
method is known to converge exponentially quickly to the leading eigenvector.
However, the leading factor in the convergence rate is the ratio of
$\lambda_1/\lambda_2$, and hence providing a finite-time guarantee of this
method would require a non-asymptotic analysis of the rate of convergence of
the \emph{second} singular value of finite sections of a Toeplitz operator.

\subsection{Norms of random polynomials}

A significant portion of our analysis relies on bounding the norms of random
trigonometric polynomials of the form $Q(z) = \sum_{k=0}^{r-1} \varepsilon_k z^{k}$.  The
study of the supremum norm of random finite degree polynomials was first
initiated by Salem and Zygmund~\cite{salem54}, who studied the setting where the
coefficients are drawn from a symmetric Bernoulli distribution supported on
$\{\pm 1\}$.  Later, Kahane~\cite{kahane94} proved that when the coefficients
are distributed as an isotropic Gaussian, then with probability at least $1 -
\delta$, $\norm{Q}_\infty \leq O(\sqrt{r\log(r/\delta)})$.  More recently,
Meckes~\cite{meckes07} extended this result to hold for independent
sub-Gaussian random variables by employing standard tools from probability in
Banach spaces.
In Section~\ref{sec:fir_sysid}, we extend these results to the case when the
coefficients follow a non-isotropic Gaussian distribution.  This is important
because it allows us to reduce the overall error of our estimate by using
non-isotropic covariance matrices from experiment design.


\section{System Identification of Finite Impulse Responses}
\label{sec:fir_sysid}

Recall from Section~\ref{sec:intro:results} that we are given query access to
$G$ via the form
\begin{align*}
    Y_{u,T} = (g \ast u)_{k=0}^{T-1} + \xi \:, \:\: \xi \sim N(0, \sigma^2 I_{T}) \:,
\end{align*}
where $u \in
\mathcal{U}$ for some fixed set $\mathcal{U}$.  Therefore, the ratio of some
measure of the size of $\mathcal{U}$ to $\sigma$ serves as the signal-to-noise
(SNR) ratio for our setting.
In what follows, we will always assume $\mathcal{U}$ is a unit $\ell_p$-ball
for $p \in [1, \infty]$.

Fix a set of $m$ inputs $u_1, ..., u_m \in \mathcal{U}$.
Given a realization of $\{Y_{u_k,T}\}_{k=1}^{m}$,
we can estimate the first $T$ coefficients $\{g_k\}_{k=0}^{T-1}$
of $G(z) = \sum_{k=0}^{\infty} g_k z^{-k}$ via ordinary least-squares (OLS).
Calling the vector $Y := (Y_{u_1,T}, ..., Y_{u_m,T}) \in \R^{Tm}$,
it is straightforward to show that the least squares estimator
$\widehat{g}_{0:T-1}$ is given
by
\begin{align*}
  \widehat{g}_{0:T-1} := \begin{bmatrix}
    \widehat{g}_0 \\
    \widehat{g}_1 \\
    \vdots \\
    \widehat{g}_{T-1} \\
  \end{bmatrix}
  = (Z^\T Z)^{-1} Z^\T Y \:, \:\:
  Z := \begin{bmatrix} \Toep(u_1) \\ \vdots \\ \Toep(u_m) \end{bmatrix} \in \R^{Tm \times T} \:.
\end{align*}
Let us clarify the $\Toep(u)$ notation.
For a vector $u \in \R^T$, $\Toep(u)$ is the $T \times T$ lower-triangular
Toeplitz matrix where the first column is equal to $u$.
Later on, we will use the notation $\Toep_{a \times b}(u)$, where
$a,b$ are positive integers. This is to be interpreted as
the upper left $a \times b$ section of the semi-infinite
lower-triangular Toeplitz matrix form by treating $u$ as a zero-padded
sequence in $\R^\mathbb{N}$.

Above, we assume the matrix $Z^\T Z$ is invertible, which we will
ensure in our analysis.
From $\widehat{g}_{0:T-1}$, we form the estimated finite impulse response
$\widehat{G}_r$ for any $r \leq T$ as $\widehat{G}_r(z) := \sum_{k=0}^{r-1} \widehat{g}_k
z^{-k}$.  The Gaussian output noise assumption means that the error vector
$\widehat{g}_{0:T-1} - g_{0:T-1}$ is distributed $N(0, \sigma^2 (Z^\T Z)^{-1})$, and hence $\widehat{G}_r - G_r$ is equal in
distribution to the random polynomial $Q(z) = \sum_{k=0}^{r-1} \varepsilon_k
z^{-k}$ with $\varepsilon \sim N(0, \sigma^2 E_r (Z^\T Z)^{-1} E_r^\T)$,
where $E_r := \begin{bmatrix} I_r & 0_{r \times (T-r)} \end{bmatrix} \in \R^{r \times T}$.
Here, $G_r(z) := \sum_{k=0}^{r-1} g_k z^{-k}$ is the length-$r$ FIR
truncation of $G$.
Since the covariance matrix will play a critical role in our analysis to follow,
we introduce the notation
\begin{align}
  \Sigma(u) := \sum_{k=1}^{m} \Toep(u_k)^\T \Toep(u_k) \:, \label{eq:def_cov_matrix}
\end{align}
where $m$ will be clear from context.
We will also use the shorthand notation $[M]_{[r]}$,
to refer to the $r \times r$ matrix $E_r M E_r^\T$ for any $T \times T$ matrix $M$.

The roadmap for this section is as follows.
In Section~\ref{sec:sysid:concentration}, we characterize
the behavior of the random quantity $\norm{Q}_\infty$ as a function of the
covariance matrix $\sigma^2 \Sigma(u)^{-1}$ and the polynomial degree $r$.
Next, we study in Section~\ref{sec:sysid:aopt}
the problem of experiment design for choosing the best inputs $u_1, ...,
u_m$ to minimize the error $\norm{Q}_\infty$.
Using these results, we give upper bounds for FIR identification
with $\ell_p$-constrained inputs in Section~\ref{sec:sysid:upper_bounds}.
We then combine these results and prove in Section~\ref{sec:sysid:main_upper_bounds}
the main results from
Theorem~\ref{thm:main_result_ltwo} and
Theorem~\ref{thm:main_result_linfty}.
Finally, we prove the minimax risk lower bound from
Theorem~\ref{thm:main_result_lower_bound} in Section~\ref{sec:sysid:main_lower_bound}.

\paragraph{Process noise.}

Before we begin our analysis, we note that our upper bounds easily
extend to the case where process noise enters the system
through the same channel as the input.
Specifically, suppose the input signal is corrupted by $\zeta \sim N(0, \sigma_n^2 I_{T})$
which is independent of the output noise $\xi$,
and instead of observing $Y_{u,T}$ we observe
\begin{align*}
  \widetilde{Y}_{u,T} = (g \ast \widetilde{u})_{k=0}^{T-1} + \xi \:, \:\: \widetilde{u} := u + \zeta \:.
\end{align*}
In this setting, the error vector $\widehat{g}_{0:T-1} - g_{0:T-1}$ of the least-squares
estimator on $\{\widetilde{Y}_{u_k,T}\}_{k=1}^{m}$ is distributed
\begin{align*}
  N(0, \Lambda) \:, \:\: \Lambda := (Z^\T Z)^{-1} Z^\T( \sigma_n^2 \Toep_{T \times T}(g)\Toep_{T \times T}(g)^\T + \sigma^2 I_{T}) Z (Z^\T Z)^{-1} \:.
\end{align*}
Since $g$ is the impulse response of a stable system,
$\norm{\Toep_{T \times T}(g)} \leq \hinfnorm{G}$, and therefore
$\Lambda \preccurlyeq (\sigma_n^2 \hinfnorm{G}^2 + \sigma^2) (Z^\T Z)^{-1}$.
Thus, the upper bounds carry over to this process noise setting with
the variable substitution $\sigma^2 \gets \sigma_n^2 \norm{G}_\infty^2 + \sigma^2$.
The modification to the lower bounds in this setting is more delicate, and we leave
this to future work.

\subsection{A concentration result for the error polynomial}
\label{sec:sysid:concentration}

We first address the behavior of the error $\norm{Q}_\infty$.    Our main tool is
a discretization result from Bhaskar et al.~\cite{bhaskar12}:
\begin{lem}[Bhaskar et al.~\cite{bhaskar12}]
\label{lem:discretization}
Let $Q(z) := \sum_{k=0}^{r-1} \varepsilon_k z^{-k}$, where $\varepsilon_k \in \C$.
For any $N \geq 4\pi r$,
  \begin{align*}
    \norm{Q}_\infty \leq \left(1 + \frac{4\pi r}{N}\right) \max_{k=0, ..., N-1} \abs{ Q( e^{j2\pi k/N} )} \:.
  \end{align*}
\end{lem}

Lemma~\ref{lem:discretization} immediately reduces
controlling the $\mathcal{H}_\infty$-norm of a finite-degree polynomial to
controlling the maxima of a finite set of points on the torus.
Hence, upper bounding the expected value of $\norm{Q}_\infty$ and showing
concentration is straightforward.
Before we state the result, we define some useful notation
which we will use throughout this section.
For a $z \in \C$, define the vector of monomials $\varphi(z)$ as
\begin{align}
    \varphi(z) := (1, z, z^2, ..., z^{r-1}) \in \C^r \:, \:\: \varphi_1(z) := \Re\{\varphi(z)\} \:, \:\: \varphi_2(z) := \Im\{\varphi(z)\} \:, \label{eq:monomials_definition}
\end{align}
where the length $r$ will be implicit from context.
\begin{lem}
\label{lem:concentration_of_error}
Let $\varepsilon \sim N(0, V)$ where $\varepsilon \in \R^r$, and put $Q(z) = \sum_{k=0}^{r-1} \varepsilon_k z^{-k}$.
Define for $\ell=1,2$,
\begin{align}
    \eta_\ell^2 := \sup_{z \in \mathbb{T}} \: \varphi_\ell(z)^\T V \varphi_\ell(z) \:. \label{eq:effective_variance}
\end{align}
We have that
\begin{align}
  \label{eq:expected_value_bound}
  \E\norm{Q}_\infty \leq 4\sqrt{2} \eta \sqrt{\log(8 \pi r)} \:,
\end{align}
where $\eta := \max(\eta_1, \eta_2)$.
Furthermore, with probability at least $1-\delta$, we have
  \begin{align}
      \label{eq:prob_bound}
      \norm{Q}_\infty \leq 4 \sqrt{2} \eta (\sqrt{\log(8\pi r)}  + \sqrt{\log(2/\delta)}) \:.
  \end{align}
\end{lem}
\begin{proof}
Set $N = 4\pi r$ and invoke Lemma~\ref{lem:discretization}
to conclude that
    \begin{align*}
        \norm{Q}_\infty &\leq 2 \max_{k=0, ..., 4\pi r - 1} \abs{Q(e^{j k/2r})} \\
        &= 2 \max_{k=0, ..., 4\pi r - 1} \abs{\ip{\varphi(e^{j k/2r})}{\varepsilon}} \\
        &\leq 2 \max_{k=0, ..., 4\pi r - 1} \abs{\ip{\varphi_1(e^{j k/2r})}{\varepsilon}} + 2 \max_{k=0, ..., 4\pi r - 1} \abs{\ip{\varphi_2(e^{j k/2r})}{\varepsilon}} \:.
    \end{align*}
We first prove \eqref{eq:expected_value_bound}
by bounding $\E \max_{k=0, ..., 4\pi r - 1}\abs{\ip{\varphi_\ell(e^{j k/2r})}{\varepsilon}}$ for $\ell=1,2$.
For a fixed $k$, we have that
$\ip{\varphi_\ell(e^{j k/2r})}{\varepsilon} \sim N(0, \varphi_\ell(e^{j k/2r})^\T V \varphi_\ell(e^{j k/2r}))$.
By standard results for expected maxima of Gaussian random variables,
we have
    \begin{align*}
        \E \max_{k=0, ..., 4\pi r - 1}\abs{\ip{\varphi_\ell(e^{j k/2r})}{\varepsilon}} &\leq \max_{k=0, ..., 4\pi r - 1} \sqrt{\varphi_\ell(e^{j k/2r})^\T V \varphi_\ell(e^{j k/2r})} \sqrt{2\log(8\pi r)} \\
        &\leq \eta \sqrt{2 \log(8\pi r)} \:.
    \end{align*}
This yields \eqref{eq:expected_value_bound}.

For \eqref{eq:prob_bound}, using standard concentration results for
suprema of Gaussian processes (see e.g. \cite{boucheron16}), we have that with probability at least $1-\delta$,
    \begin{align*}
        \max_{k=0, ..., 4\pi r - 1}\abs{\ip{\varphi_\ell(e^{j k/2r})}{\varepsilon}} &\leq \E \max_{k=0, ..., 4\pi r - 1}\abs{\ip{\varphi_\ell(e^{j k/2r})}{\varepsilon}} + \eta \sqrt{2\log(1/\delta)} \\
        &\leq \eta \sqrt{2\log(8\pi r)} + \eta \sqrt{2\log(1/\delta)} \:.
    \end{align*}
The claim \eqref{eq:prob_bound} now follows from a union bound.
\end{proof}
Note that when $V = I$, $\eta^2 \leq r$ which recovers the known results
from \cite{kahane94} up to constants.
Furthermore, when $V$ is diagonal, $\eta^2 \leq \Tr(V)$.
We will exploit this result in the sequel.

\subsection{Experiment design}
\label{sec:sysid:aopt}

We now consider the problem of choosing a set of inputs $u \in \mathcal{U}$
in order to minimize the expected error of the residual polynomial.
Fixing the number of inputs $m$, the input constraint set $\mathcal{U}$,
and recalling the definition of the covariance $\Sigma(u)$ from \eqref{eq:def_cov_matrix},
the optimal experiment design problem is
\begin{align}
\label{eq:opt_design}
    \operatorname*{minimize}_{u_1, ..., u_m \in \mathcal{U}} \: \E_{\varepsilon \sim N(0, [\Sigma(u)^{-1}]_{[r]})} \norm{Q}_\infty \:.
\end{align}
In \eqref{eq:opt_design} and the sequel, if the covariance matrix $\Sigma(u)$ is not invertible then we assign the function value $+\infty$.
Problem \eqref{eq:opt_design} is difficult to solve as written because the expected value does not have
a form which is easy to work with computationally.
The following design problem provides a good approximation of \eqref{eq:opt_design}.
Let $\{z_1, ..., z_s\} \subseteq \mathbb{T}$ denote a grid of $s$ points on $\mathbb{T}$. Consider the problem
\begin{align}
\label{eq:opt_design_grid}
   \operatorname*{minimize}_{u_1, ..., u_m \in \mathcal{U}} \: \max_{\substack{1 \leq k \leq s \\ \ell=1,2}} \varphi_\ell(z_k)^\T [\Sigma(u)^{-1}]_{[r]} \varphi_\ell(z_k) \:.
\end{align}
The objective \eqref{eq:opt_design_grid} minimizes the maximum pointwise variance of $Q(z)$ over all points on the
grid $\{z_1, ..., z_s\}$.
If the grid is uniformly spaced and $s \geq 4\pi r$, then by Lemma~\ref{lem:discretization}
we can interpret \eqref{eq:opt_design_grid} as minimizing an upper bound
to the objective function in \eqref{eq:opt_design}, since
\begin{align}
  \E_{\varepsilon \sim N(0, [\Sigma(u)^{-1}]_{[r]})} \norm{Q}_\infty &\leq \left(1 + 4\pi r /s \right) \E \max_{1 \leq k \leq s} \abs{\ip{\varphi(z_k)}{\varepsilon}} \nonumber \\
    &\leq \left(1 + 4\pi r /s \right) \sqrt{2\log(2s)} \sum_{\ell=1}^{2} \max_{1 \leq k \leq s} \sqrt{\varphi_\ell(z_k)^\T [\Sigma(u)^{-1}]_{[r]} \varphi_\ell(z_k)} \label{eq:min_upper_bound} \:.
\end{align}
However, \eqref{eq:opt_design_grid} is non-convex in the $u_i$'s.
A convex version of the problem can be written by choosing $m_0$ inputs
$u_1, ..., u_{m_0} \in \mathcal{U}$ and solving the semidefinite program (SDP)
\begin{align}
\label{eq:opt_design_sdp}
    \operatorname*{minimize}_{\lambda \in \R^{m_0}} \: \max_{\substack{1 \leq k \leq s \\ \ell=1,2}} \varphi_\ell(z_k)^\T [\Sigma^{-1}]_{[r]} \varphi_\ell(z_k) \:\: \mathrm{s.t.} \:\: \Sigma = \sum_{i=1}^{m_0} \lambda_i \mathrm{Toep}(u_i)^\T\mathrm{Toep}(u_i) \:, \:\: \lambda^\T \ind = 1 \:, \:\: \lambda \geq 0 \:.
\end{align}
\eqref{eq:opt_design_sdp} is a convex program and can be solved with any off-the-shelf solver
such as MOSEK~\cite{mosek}.

We now study two special cases of $\mathcal{U}$ to show how
input constraints can affect design.
We first observe that when $\Sigma(u)$ is diagonal, continuing the estimates
from \eqref{eq:min_upper_bound}, we have the following
upper bound which holds since $\norm{\varphi_\ell(z)}_\infty \leq 1$,
\begin{align}
  \E_{\varepsilon \sim N(0, [\Sigma(u)^{-1}]_{[r]})} \norm{Q}_\infty \leq \left(1 + 4\pi r /s \right) 2\sqrt{2\log(2s)\Tr([\Sigma(u)^{-1}]_{[r]})}  \:. \label{eq:min_upper_bound_diag}
\end{align}
Even though \eqref{eq:min_upper_bound_diag} only holds when $\Sigma(u)$ is diagonal, it motivates us to
consider the standard $A$-optimal design problem
\begin{align}
\label{eq:opt_design_aopt}
    \operatorname*{minimize}_{u_1, ..., u_m \in \mathcal{U}} \: \Tr([\Sigma(u)^{-1}]_{[r]}) \:.
\end{align}
An advantage of \eqref{eq:opt_design_aopt} versus \eqref{eq:opt_design_grid}
is that the reduced complexity of the objective function
allows us to make statements about optimality when $\mathcal{U}$ is an $\ell_p$-ball.
The analogous SDP relaxation of \eqref{eq:opt_design_aopt}, similar to
\eqref{eq:opt_design_sdp}, is also more efficient to implement in practice
for more general $\mathcal{U}$'s.

Let $F^*_p(T, r)$ denote the optimal value of \eqref{eq:opt_design_aopt} with
$\mathcal{U} = B_p^T$ for $p\in [1, \infty]$, where $B_p^T$ denotes the unit
$\ell_p$-ball in $\R^T$.  We will always assume $T \geq r$.
It is not hard to show that $F^*_p(T, r)$ is finite and the value is
attained (and hence $\Sigma(u)$ at the optimum is invertible).
We will now show the following statements about \eqref{eq:opt_design_aopt}:
\begin{enumerate}[(a)]
  \item When $p \in [1,2]$, the optimal solution is to set $u_i = e_1$, $i=1, ..., m$.
  \item When $p \in (2, \infty]$, we can solve \eqref{eq:opt_design_aopt}
    to within a factor of two of optimal by convex programming.
  \item When $p = \infty$, we can give an exact closed form solution for
    \eqref{eq:opt_design_aopt} in the special case when
    $r = 2^n$ with $n \geq 0$, $T= 2^k r$ with $k \geq 1$,
		and $m$ is a multiple of $T$.
\end{enumerate}

\subsubsection{\texorpdfstring{$A$-optimal design for $\ell_p$-balls}{A-optimal design for l-p-balls}}
\label{sec:sysid:aopt_lp}

We first prove a lower bound on the optimal objective value $F^*_p(T, r)$.
To do this, we use the following linear algebra fact.
\begin{lem}
\label{lem:trace_diag_ineq}
Let $A$ be an $n \times n$ positive definite matrix. We have that
    \begin{align*}
        \Tr(A^{-1}) \geq \sum_{i=1}^{n} A_{ii}^{-1} \:.
    \end{align*}
\end{lem}
\begin{proof}
This proof is due to Mateusz Wasilewski~\cite{wasilewski12}.

By the Schur-Horn theorem, we know that the eigenvalues of $A$ majorize the diagonal of $A$, i.e.
    \begin{align*}
        \sum_{i=1}^{k} A_{ii} \leq \sum_{i=1}^{k} \lambda_i(A) \:, \:\: k=1, ..., n \:, \:\: \sum_{i=1}^{n} A_{ii} = \sum_{i=1}^{n} \lambda_i(A) \:.
    \end{align*}
The function $x \mapsto 1/x$ is convex for $x > 0$.
This allows us to apply Karamata's inequality, from which the claim immediately follows.
\end{proof}
Lemma~\ref{lem:trace_diag_ineq} immediately yields
an optimization problem which lower bounds the optimal value $F^*_p$.
\begin{lem}
\label{lem:aopt_lower_bound}
For all $p \in [2, \infty]$, we have that
\begin{align}
    F^*_p(T, r) \geq \frac{1}{m} \inf_{\substack{w \in \R^T \\ \norm{w}_{p/2} \leq 1, w \geq 0}} \sum_{i=1}^{r} \left[ \sum_{\ell=1}^{T-r+i} w_\ell \right]^{-1} := \frac{1}{m} D_p(T, r) \:.
\end{align}
\end{lem}
\begin{proof}
We have that
\begin{align*}
    F^*_p(T, r) &= \inf_{\substack{u_1, ..., u_m \in \R^T \\ \norm{u_i}_p \leq 1 \\ \det{\Sigma(u)} \neq 0}} \Tr([\Sigma(u)^{-1}]_{[r]})
          \stackrel{(a)}{\geq} \inf_{\substack{u_1, ..., u_m \in \R^T \\ \norm{u_i}_p \leq 1 \\ \det{\Sigma(u)} \neq 0}} \Tr( [\Sigma(u)]_{[r]}^{-1}    )
          \stackrel{(b)}{\geq} \inf_{\substack{u_1, ..., u_m \in \R^T \\ \norm{u_i}_p \leq 1 \\ \det{\Sigma(u)} \neq 0}} \sum_{i=1}^{r} \Sigma(u)_{ii}^{-1} \\
          &\stackrel{(c)}{\geq} \inf_{\substack{u_1, ..., u_m \in \R^T \\ \norm{u_i}_p \leq 1}} \sum_{i=1}^{r} \Sigma(u)_{ii}^{-1}
          = \inf_{\substack{u_1, ..., u_m \in \R^T \\ \norm{u_i}_p \leq 1}} \sum_{i=1}^{r} \left[ \sum_{k=1}^{m} \sum_{\ell=1}^{T-r+i} (u_k)^2_{\ell} \right]^{-1}
          \stackrel{(d)}{=} \frac{1}{m} \inf_{\substack{u \in \R^T \\ \norm{u}_p \leq 1}} \sum_{i=1}^{r} \left[ \sum_{\ell=1}^{T-r+i} u^2_{\ell} \right]^{-1} \\
          &\stackrel{(e)}{=} \frac{1}{m} \inf_{\substack{w \in \R^T \\ \norm{w}_{p/2} \leq 1, w \geq 0}} \sum_{i=1}^{r} \left[ \sum_{\ell=1}^{T-r+i} w_\ell \right]^{-1} \:.
\end{align*}
Above,
(a) follows since for any positive definite matrix $M$, $[M^{-1}]_{[r]} \succcurlyeq [M]_{[r]}^{-1}$
(see e.g. Equation (3.2.27) of~\cite{horn05}), combined with the fact that
trace is operator monotone,
(b) follows from Lemma~\ref{lem:trace_diag_ineq},
(c) follows since the infimum is over a larger set,
(d) follows by the symmetry of the objective function with respect to $u_1, ..., u_m$, and
(e) follows by a simple reparameterization.
\end{proof}
Before we proceed, a couple of remarks are in order.  First, it is clear that
$D_\infty(T,r) = H_{T} - H_{T-r} = \Theta(\log(T/(T-r)))$, where $H_n$ is the $n$-th
Harmonic number.  This is achieved by setting $w = \ind$.  Second, it is
straightforward to show (e.g. by the KKT conditions) that $D_2(T,r) = r$, which
is achieved by setting $w = e_1$. When $p \in (2, \infty)$, the optimal $u$
will be an interpolation between $w=\ind$ and $w=e_1$.  To the best of our
knowledge, there is no simple closed form formula for $w$ in the general case.
However, we prove the following upper and lower bound on $D_p(T,r)$.
\begin{lem}
\label{lem:dp_approx}
When $p > 2$, we have that
\begin{align*}
    (T-r+1)^{2/p} (H_{T} - H_{T-r}) \leq D_p(T, r) \leq T^{2/p} (H_{T} - H_{T-r}) \:.
\end{align*}
\end{lem}
\begin{proof}
Set $p' = p/2$ and let $q'$ denote its conjugate pair.
The upper bound follows by plugging in the feasible vector $w = \ind_{T}/\norm{\ind_{T}}_{p'}$.
The lower bound follows by the chain of inequalities
\begin{align*}
    D_p(T,r) &= \inf_{\substack{w \in \R^T \\ \norm{w}_{p'} \leq 1, w \geq 0}} \sum_{i=1}^{r} \left[ \sum_{\ell=1}^{T-r+i} w_\ell \right]^{-1}
    \geq \sum_{i=1}^{r} \inf_{\substack{w \in \R^T \\ \norm{w}_{p'} \leq 1, w \geq 0}} \left[ \sum_{\ell=1}^{T-r+i} w_\ell \right]^{-1} \\
    &= \sum_{i=1}^{r} \left[  \sup_{\substack{w \in \R^T \\ \norm{w}_{p'} \leq 1, w \geq 0}} \sum_{\ell=1}^{T-r+i} w_\ell \right]^{-1}
    = \sum_{i=1}^{r} \frac{1}{\norm{ \ind_{T-r+i} }_{q'}} \\
    &= \sum_{i=1}^{r} \left( \frac{1}{T-r+i} \right)^{1 - 1/p'}
    = \sum_{i=T-r+1}^{T} \left( \frac{1}{i} \right)^{1 - 1/p'} \\
    &= \sum_{i=T-r+1}^{T} \frac{1}{i} i^{1/p'} \geq (T-r+1)^{1/p'} \sum_{i=T-r+1}^{T} \frac{1}{i} \\
    &= (T-r+1)^{1/p'} (H_{T} - H_{T-r}) \:.
\end{align*}
\end{proof}
Lemma~\ref{lem:dp_approx} implies that if e.g. $T=2r$, then
$D_p(2r, r) = \Theta( r^{2/p} \log{r} )$.
We now use the lower bound on $F^*_2(T, r)$ to show that
in the regime $p \in [1,2]$, we can solve the problem \eqref{eq:opt_design_aopt}
exactly.
\begin{lem}
\label{lem:lp_one_two_exact}
For $p \in [1,2]$, we have that $F^*_p(T, r) = \frac{r}{m}$, which is achieved by setting
$u_1 = ... = u_m = e_1$.
\end{lem}
\begin{proof}
For $u_1 = ... = u_m = e_1$, it is immediate to verify that $\Sigma(u) = m I$, which
yields $\Tr(\Sigma(u)^{-1}) = \frac{r}{m}$.
Hence by Lemma~\ref{lem:aopt_lower_bound},
for the $p=2$ case, $\frac{r}{m} \geq F^*_2(T,r) \geq \frac{1}{m} D_2(T,r) = \frac{r}{m}$.
On the other hand, since $e_1$ is feasible for all $\ell_p$-balls and
since $B_p^T \subseteq B_2^T$ for $p \in [1,2]$,
$\frac{r}{m} \geq F^*_p(T, r) \geq F^*_2(T, r) = \frac{r}{m}$.
\end{proof}
The remainder of this section will focus on the regime when $p > 2$.
In this case, $D_p(T, r)$ is a convex program, and for each fixed $T,r,p$, one can quickly
solve for the optimal value numerically.
Our goal now is to use utilize this numerical solution for approximating
the experiment design problem.
Specifically, we show how we can choose $m$ inputs of length $T$
via convex programming such that
$\Tr([\Sigma(u)^{-1}]_{[r]}) = \frac{2}{m} D_p(T, r)$. Recall by Lemma~\ref{lem:aopt_lower_bound},
that $F^*_p(T, r) \geq \frac{1}{m} D_p(T, r)$. Hence, we can always recover a solution
that is within a factor of two of the optimal solution by convex programming.
The loss of a factor of two is due to the real-valued nature of the inputs:
if complex-valued inputs were allowed then our solution would be exact.

Before we can state the result, we need a few auxiliary lemmas and notation.
For two vectors $x,y \in \C^T$, we let the
notation $x \odot y = (x_1 y_1, ..., x_T y_T) \in \C^T$ denote
the vector formed by the element wise product of the entries.
The next lemma states that when we form a (Hermitian) covariance matrix
from weighted sinusoidal inputs, the resulting covariance matrix
has a simple diagonal structure.
\begin{lem}
\label{lem:general_toep_structure}
Fix positive integers $T, n$ with $n \geq T$.
Define $z_i = e^{2\pi j i / n}$ for $i=0, ..., n-1$.
Fix any weighting $w \in \C^T$.
We have that
\begin{align}
  M := \sum_{i=0}^{n-1} \Toep( w \odot \varphi(z_i) )^*\Toep( w \odot \varphi(z_i) ) = n \diag\left( \sum_{i=1}^{T} \abs{w_i}^2, \sum_{i=1}^{T-1} \abs{w_i}^2, ..., \abs{w_1}^2 \right) \:,
\end{align}
\end{lem}
\begin{proof}
To ease indexing notation, we zero index the coefficients of $w$ in the proof.
Define the matrix $M_z := \Toep(w \odot \varphi(z))^*\Toep(w \odot \varphi(z))$.
Let us compute the entries of $(M_z)_{k,\ell=0}^{T-1}$ for $z = e^{j \theta}$.
Assuming $k < \ell$, we have that
\begin{align*}
  (M_z)_{k,\ell} = \left(\sum_{m=0}^{T - (\ell + 1)} \overline{w}_{\ell-k+m} w_{m}\right)  e^{-j \theta(\ell-k)} \:.
\end{align*}
Now summing over the $z_i$'s,
\begin{align*}
  M_{k,\ell} = \sum_{i=0}^{n-1} (M_{z_i})_{k,\ell} &= \left(\sum_{m=0}^{T - (\ell + 1)} \overline{w}_{\ell-k+m} w_{m}\right)  \sum_{i=0}^{n-1} e^{-j \frac{2\pi j i}{n} (\ell - k)} \\
  &=  \left(\sum_{m=0}^{T - (\ell + 1)} \overline{w}_{\ell-k+m} w_{m}\right) \frac{1 - e^{-2\pi j(\ell-k)}}{1 - e^{- \frac{2\pi j}{n}(\ell-k)}} = 0 \:.
\end{align*}
Above, the last inequality holds because we assumed $n \geq T$.
Since $M_z$ is Hermitian,
$(M_z)_{\ell,k} = \overline{(M_z)_{k, \ell}} = 0$, and therefore
we have that $M$ is a diagonal matrix.
The diagonal entries are simply given by
\begin{align*}
  M_{kk} = n \sum_{i=1}^{T-k+1} \abs{w_i}^2 \:.
\end{align*}
The claim now follows.
\end{proof}
The weighting vector $w$ from Lemma~\ref{lem:general_toep_structure} will
be useful when we consider inputs constrained by general $\ell_p$-balls.
We now state a lemma which will allow us to work with purely real-valued
input signals by splitting the signals up into the real and imaginary parts,
at an expense of a factor of two. This is where our
sub-optimality enters in.
\begin{lem}
\label{lem:real_inputs}
Fix positive integers $m, n, T$, and suppose that $m = 2n$ and $n \geq T$.
Also fix any $w \in \R^T$.
Let $z_i = e^{2\pi j i/n}$ for $i=0, ..., n-1$.
Define the vectors $u_0, ..., u_{n-1}$ and $u_{n}, ..., u_{2n-1}$ as
\begin{align*}
  u_i = \Re\{w \odot \varphi(z_i)\} \:, \:\: u_{n + i} = \Im\{w \odot \varphi(z_i)\} \:, \:\: i = 0, ..., n-1 \:.
\end{align*}
We have that
\begin{align}
  \Sigma(u) = \frac{m}{2} \diag\left( \sum_{i=1}^{T} \abs{w_i}^2, \sum_{i=1}^{T-1} \abs{w_i}^2, ..., \abs{w_1}^2 \right) \:. \label{eq:cov_real_inputs}
\end{align}
\end{lem}
\begin{proof}
We observe that for any complex matrix $X$,
\begin{align*}
\Re\{X^* X\} = \Re\{X\}^\T \Re\{X\} + \Im\{X\}^\T \Im\{X\} \:.
\end{align*}
Furthermore, $\Re\{\Toep(u)\} = \Toep(\Re\{u\})$
and similarly $\Im\{\Toep(u)\} = \Toep(\Im\{u\})$.
Hence,
\begin{align*}
  &\Re\{\Toep(w \odot \varphi(z_i))^* \Toep(w \odot \varphi(z_i))\} \\
    &\qquad= \Toep(\Re\{w \odot \varphi(z_i)\})^\T\Toep(\Re\{w \odot \varphi(z_i)\})
     + \Toep(\Im\{w \odot \varphi(z_i)\})^\T\Toep(\Im\{w \odot \varphi(z_i)\}) \:.
\end{align*}
Therefore, we have the identity
\begin{align*}
  \Sigma(u) = \Re\left\{ \sum_{i=0}^{n-1} \Toep(w \odot \varphi(z_i))^* \Toep(w \odot \varphi(z_i)) \right\} \:.
\end{align*}
The claim now follows from Lemma~\ref{lem:general_toep_structure}.
\end{proof}
We are finally ready to define the input ensemble which approximately solves the $A$-optimal
experimental design problem in the $p \in (2,\infty]$ regime.
For any $p \in (2, \infty]$, define the vector
$w_p(T, r) \in \R^T$ as any vector which achieves the infimum in
the optimization problem defining $D_p(T, r)$ before re-parameterization, i.e.
$w_p(T, r)$ satisfies $\norm{w_p(T, r)}_p \leq 1$ and
\begin{align*}
    \sum_{i=1}^{r} \left[ \sum_{\ell=1}^{T-r+i} (w_p(T,r))_\ell^2 \right]^{-1} = D_p(T, r) \:.
\end{align*}
Recall that $w_p(T, r)$ can be solved for numerically via convex optimization.
\begin{lem}
\label{lem:aopt_upper_bound}
Fix positive integers $m, n, T$, and suppose that $m = 2n$ and $n \geq T$.
Let $z_i = e^{2\pi j i/n}$ for $i=0, ..., n-1$.
Define the vectors $u_0, ..., u_{n-1}$ and $u_n, ..., u_{2n-1}$ as
\begin{align*}
    u_i = \Re\{w_p(T, r) \odot \varphi(z_i)\} \:, \:\: u_{n+i} = \Im\{w_p(T, r) \odot \varphi(z_i)\} \:, \:\: i = 0, ..., n-1 \:.
\end{align*}
We have that the covariance matrix $\Sigma(u)$ satisfies
\begin{align*}
    \Tr([\Sigma(u)^{-1}]_{[r]}) = \frac{2}{m} D_p(T, r) \:.
\end{align*}
\end{lem}
\begin{proof}
This is a direct consequence of Lemma~\ref{lem:real_inputs}.
First, we check that both $\norm{u_i}_p \leq 1$ and $\norm{u_{n+i}}_p \leq 1$.
To check $u_i$, note that
\begin{align*}
    \norm{u_i}_p = \norm{\Re\{ w_p(T, r) \odot \varphi(z_i) \}}_p \leq \norm{w_p(T, r) \odot \varphi(z_i)}_p = \norm{w_p(T, r)}_p \leq 1 \:.
\end{align*}
A similar calculation holds for $u_{n+i}$.
Hence from Lemma~\ref{lem:real_inputs}, specifically \eqref{eq:cov_real_inputs},
we have that the covariance matrix $\Sigma(u)$ is diagonal, invertible, and
satisfies $\Tr([\Sigma(u)^{-1}]_{[r]}) = \frac{2}{m} D_p(T, r)$.
\end{proof}
At this point we are nearly done, since we have shown that
\eqref{eq:opt_design_aopt} can be solved to within a factor of two.
We conclude this section by showing that
when $p=\infty$,
$r = 2^n$ with $n \geq 0$, $T= 2^k r$ with $k \geq 1$,
and $m$ is a multiple of $T$, we can
we can remove the factor of two sub-optimality and
exactly solve \eqref{eq:opt_design_aopt}.

\paragraph{Hadamard construction.}
Our construction is based on the Hadamard transform.
We will show that the optimal input vectors for $A$-design are $r$ orthogonal vectors in $\{-1,+1\}^{r}$.
We give a construction for these vectors in the following proposition.
\begin{prop}
\label{prop:orthogonal_vectors}
For $n=0,1,2,...$, there exists
$2^n$ vectors in $\{-1,+1\}^{2^n}$
that are orthogonal with respect to the standard $\ell_2$ inner product on $\R^{2^n}$.
\end{prop}
\begin{proof}
We will induct on $n=0,1,2,...$, for which the base case $n=0$ holds with $u_0=1$. Assume we have $2^n$ orthogonal vectors, in $\{-1,+1\}^{2^n}$, denoted $\{u_k\}$. Then, the $2^{n+1}$ vectors
\begin{align*}
\{\tilde u_k\}:=\bigcup_{k=0}^{2^n-1}\left\{\begin{bmatrix}u_k\\u_k\end{bmatrix},\:\begin{bmatrix}u_k\\-u_k\end{bmatrix}\right\},
\end{align*}
which reside in $\{-1,+1\}^{2^{n+1}}$, are also orthogonal.
\end{proof}
\begin{lem}
\label{lem:orthogonal_vectors_M}
The constructed orthogonal vectors $\{u_k\}_{k=0}^{2^n-1}$ specified in Proposition~\ref{prop:orthogonal_vectors} satisfy
\begin{align*}
  \Sigma(u)=\sum_{k=0}^{2^n-1}\Toep{(u_k)}^\T \Toep{(u_k)} = 2^n \diag(2^n, 2^n-1, 2^n-2, \ldots, 1)\:.
\end{align*}
\end{lem}
\begin{proof}
This follows from straightforward manipulations shown in Appendix~\ref{proof:lem_orthogonal_vectors_M}.
\end{proof}
Combining Lemma~\ref{lem:aopt_lower_bound} and
Lemma~\ref{lem:orthogonal_vectors_M} implies that the construction from Proposition~\ref{prop:orthogonal_vectors}
is optimal for \eqref{eq:opt_design_aopt}.

\subsection{Upper bounds on FIR identification with \texorpdfstring{$\ell_p$}{lp}-constrained inputs}
\label{sec:sysid:upper_bounds}

Combining the results from Section~\ref{sec:sysid:concentration}
and Section~\ref{sec:sysid:aopt},
we now prove an upper bound on length-$r$ FIR identification when the inputs are
$\ell_p$-constrained.
\begin{lem}
\label{lem:fir_upper_bounds}
Fix positive integers $m$ and $r$, and set $T = 2r$.
Consider the input ensemble
$u_1 = ... = u_m = e_1$ when $p \in [1,2]$, or the input ensemble
defined in Lemma~\ref{lem:aopt_upper_bound} when $p \in (2, \infty]$
(with additional restrictions on $m$, $r$ in this case).
Let $\widehat{G}_r$ denote the length-$r$ FIR estimate derived from
least-squares, and let $G_r$ denote the length-$r$ FIR truncation of $G$.
With probability at least $1-\delta$,
\begin{align*}
    \hinfnorm{\widehat{G}_r - G_r} \leq
        \begin{cases}
            4\sqrt{2} \sigma \sqrt{\frac{r}{m}} \left(\sqrt{\log(8\pi r)} + \sqrt{\log(2/\delta)}\right) &\text{if } p \in [1,2] \\
            8\sqrt{2\log{2}} \sigma \sqrt{\frac{r^{2/p}}{m}} \left(\sqrt{\log(8\pi r)} + \sqrt{\log(2/\delta)}\right)&\text{if } p \in (2, \infty] \:.
        \end{cases}
\end{align*}
\end{lem}
\begin{proof}
From Lemma~\ref{lem:concentration_of_error}, we have that
with probability at least $1-\delta$,
\begin{align*}
    \hinfnorm{\widehat{G}_r - G_r} \leq 4\sqrt{2} \sigma \sqrt{\Tr([\Sigma(u)^{-1}]_{[r]})} \left(\sqrt{\log(8\pi r)} + \sqrt{\log(2/\delta)}\right) \:.
\end{align*}
We just need to upper bound the variance term $\Tr([\Sigma(u)^{-1}]_{[r]})$.
When $p\in[1,2]$, we know by Lemma~\ref{lem:lp_one_two_exact} that the optimal
input ensemble is the impulse response $u_i = e_1$, so we have
$\Tr([\Sigma(u)^{-1}]_{[r]}) \leq r/m$.
On the other hand, when $p \in (2, \infty]$, we know by Lemma~\ref{lem:aopt_upper_bound} that
the specified $u$'s satisfy
\begin{align*}
    \Tr([\Sigma(u)^{-1}]_{[r]}) \leq \frac{2}{m} D_p(2r, r) \leq \frac{4}{m} r^{2/p} (H_{2r}-H_{r}) \leq \frac{4\log{2}}{m} r^{2/p} \:.
\end{align*}
Above, the upper bound on $D_p(2r, r)$ follows from Lemma~\ref{lem:dp_approx}.
\end{proof}

\subsection{Proof of upper bounds for main result}
\label{sec:sysid:main_upper_bounds}

We now prove Theorem~\ref{thm:main_result_ltwo}
and Theorem~\ref{thm:main_result_linfty}.
Recall that $\widehat{G}_r$ is the estimated length-$r$ FIR approximation to $G$,
and $G_r$ is the true length-$r$ FIR truncation of $G$.
By the triangle inequality, we have the following error decomposition
into an approximation error and an estimation error
\begin{align}
    \norm{G - \widehat{G}_r}_\infty \leq \underbrace{\norm{G - G_r}_\infty}_{\text{Approx. error.}} + \underbrace{\norm{G_r - \widehat{G}_r}_\infty}_{\text{Estimation error.}} \:.
\end{align}
Hence, in order for $\norm{G - \widehat{G}_r}_\infty \leq \varepsilon$ to
hold, it suffices to have both the approximation error $\norm{G-G_r}_\infty
\leq \varepsilon/2$ and the estimation error $\norm{G_r -
\widehat{G}_r}_\infty \leq \varepsilon/2$.

The approximation error is a deterministic quantity, and its behavior
is governed by the tail decay of the impulse response coefficients $\{g_k\}_{k \geq 0}$.
In Section~\ref{sec:truncation}, we prove in Lemma~\ref{lem:bounding_C} that
as long as $r$ satisfies
\begin{align*}
  r \geq \inf_{\rho < \gamma < 1} \: \frac{1}{1-\gamma} \log\left( \frac{2\hinfnorm{G(\gamma z)}}{ \varepsilon (1-\gamma)} \right) \:,
\end{align*}
then we have $\norm{G-G_r}_\infty \leq \varepsilon/2$.

We now turn our attention to the estimation error.
For the case when $p=2$,
Lemma~\ref{lem:fir_upper_bounds}
tells us with probability at least $1-\delta$,
the estimation error satisfies
\begin{align*}
    \norm{G_r - \widehat{G}_r}_\infty \leq 4\sqrt{2}\sigma\sqrt{\frac{r}{m}} \left(\sqrt{\log(8\pi r)} + \sqrt{\log(2/\delta)}\right) \:.
\end{align*}
Setting the RHS less than $\varepsilon/2$ and solving for $m$, combining with the inequality $(a+b)^2 \leq 2(a^2 + b^2)$, we conclude
that a sufficient condition on $m$ is
\begin{align*}
    m \geq \max\left\{\frac{256 \sigma^2 r}{\varepsilon^2} \left(\log(8\pi r) + \log\left(\frac{2}{\delta}\right) \right) , \: 1 \right\} \:.
\end{align*}
This concludes the proof of Theorem~\ref{thm:main_result_ltwo}.

The proof of Theorem~\ref{thm:main_result_linfty} is nearly identical.
For the case when $p=\infty$,
Lemma~\ref{lem:fir_upper_bounds}
tells us with probability at least $1-\delta$,
the estimation error satisfies
\begin{align*}
    \norm{G_r - \widehat{G}_r}_\infty \leq 8\sqrt{2\log{2}} \sigma \sqrt{\frac{1}{m}} \left(\sqrt{\log(8\pi r)} + \sqrt{\log(2/\delta)}\right) \:.
\end{align*}
Setting the RHS less than $\varepsilon/2$ and solving for $m$,
we conclude that the sufficient condition is
\begin{align*}
  m \geq \max\left\{ \frac{ (1024 \log{2}) \sigma^2}{\varepsilon^2} \left(\log(8\pi r) + \log\left(\frac{2}{\delta}\right) \right) , \: 4r \right\} \:.
\end{align*}
This concludes the proof of Theorem~\ref{thm:main_result_linfty}.

\subsection{Proof of lower bounds for main result}
\label{sec:sysid:main_lower_bound}

We will slightly relax the estimation setting to work
with complex inputs and complex systems which will simplify the proof.
The relaxation is as follows.
Recall that $\mathscr{H}_r$ is the space of length-$r$ FIR filters,
\begin{align*}
  \mathscr{H}_r = \left\{ G(z) = \sum_{k=0}^{r-1} g_k z^{-k} : (g_k)_{k=0}^{r-1} \in \C^r \right\} \:.
\end{align*}
We fix a true $G \in \mathscr{H}_r$
parameterized by $g \in \C^r$,
and we fix $m$ signals $u_1, ..., u_m \in \C^T$ with $\norm{u_i}_p \leq 1$ for
all $i = 1, ..., m$ and $T = 2r$.
Our estimator is given a realization of the random variable $Y = (Y_1, ..., Y_m) \in \otimes_{i=1}^{m} \C^T$, where
\begin{align*}
  Y \sim \otimes_{i=1}^{m} N( \Toep_{T \times T}(g) u_i, \sigma^2 I_T ) \:.
\end{align*}
An estimator is any measurable function
$\widehat{G} : \otimes_{i=1}^{m} \C^T \longrightarrow \mathscr{H}_r$.

Let us clarify what is meant by $Y \sim N(\mu, \sigma^2 I)$ where $\mu$ is a complex vector.
We mean that the random variable $Y$ has distribution equal to
$Y = (\Re(\mu) + \sigma w_1) + j(\Im(\mu) + \sigma w_2)$ where $w_1,w_2$ are independent
$N(0, I)$ random vectors. We can equivalently treat this as the
$N((\Re(\mu), \Im(\mu)), \sigma^2 I)$ distribution over $\R^{2T}$, which means that
for $\mu_1, \mu_2 \in \C^T$,
\begin{align}
  D_{KL}( N(\mu_1, \sigma^2 I), N(\mu_2, \sigma^2 I) ) = \frac{\sigma^2}{2} \bignorm{ \begin{bmatrix} \Re(\mu_1) \\ \Im(\mu_1) \end{bmatrix} - \begin{bmatrix} \Re(\mu_2) \\ \Im(\mu_2) \end{bmatrix}}^2_{\R^{2T}} = \frac{\sigma^2}{2} \norm{\mu_1 - \mu_2}^2_{\C^{T}} \:, \label{eq:kl_div}
\end{align}
where $D_{KL}(\cdot, \cdot)$ denotes the KL-divergence.
In the sequel, we will drop the norm subscripts as they will be clear from context.

We now establish some preliminaries. First
letting $z_k = e^{2\pi j k/r}$ for $k=0, ..., r-1$ and
denoting $\varphi(z_k)$ as the length-$r$ vector of monomials,
we define the $T \times T$ matrix $M_r$ as
\begin{align}
    M_r := \frac{1}{r} \sum_{k=0}^{r-1} \Toep_{T \times T}(\varphi(z_k))^*\Toep_{T \times T}(\varphi(z_k)) \:. \label{eq:Mr}
\end{align}
It turns out this matrix will also play a central role in the lower bound analysis.
To aid the analysis, we have the following identity for $M_r$.
\begin{prop}
We have that
\begin{align}
    M_r = \diag(\underbrace{r, ..., r}_{r \textrm{ times}}, r, r-1, ..., 1) \:.
\end{align}
\end{prop}
\begin{proof}
The proof is similar to Lemma~\ref{lem:general_toep_structure}.
For $z = e^{j\theta}$ define $M_z := \Toep_{2r \times 2r}(\varphi(z))^*\Toep_{2r \times 2r}(\varphi(z))$.
When $0 \leq k < \ell \leq 2r-1$,
\begin{align*}
  (M_z)_{k,\ell} = \begin{cases} (\min(2r, k+r) - \ell) e^{-j \theta (\ell - k)} &\text{if } \ell - k < r \\
    0 &\text{o.w.}
  \end{cases} \:.
\end{align*}
Therefore, for $k < \ell$ and $\ell - k < r$,
\begin{align*}
  \sum_{i=1}^{r} (M_{z_i})_{k,\ell} = (\min(2r,k+r)-\ell) \sum_{i=0}^{r-1} e^{-j 2\pi j i(\ell-k)/r} = 0 \:,
\end{align*}
and when $k < \ell$ and $\ell - k \geq r$, $\sum_{i=1}^{r} (M_{z_i})_{k,\ell} = 0$ trivially.
Therefore, $M_r$ is a diagonal matrix.
The diagonal entries are easily computed.
\end{proof}
In light of this identity, we will be interested in upper bounds on the quantity
\begin{align}
  A_p := \sup_{\norm{u}_p \leq 1} u^* M_r u = \sup_{ \norm{u}_p \leq 1 } \sum_{k=1}^{T} \min(k, r) u_k^2 \:. \label{eq:AP}
\end{align}
We have the following calculation which establishes the desired upper bounds.
\begin{prop}
\label{prop:AP_bound}
When $p \in [1,2]$, we have
\begin{align*}
  A_p = r \:.
\end{align*}
Furthermore, when $p \in (2, \infty]$, we have the upper bound
\begin{align*}
  A_p \leq 4r^{2(p-1)/p} \:.
\end{align*}
\end{prop}
\begin{proof}
First, we treat the case when $p \in [1,2]$.
By setting $u = e_T$, we have that $A_p \geq r$.
Now, by H{\"{o}}lder's inequality,
\begin{align*}
    \sum_{k=1}^{T} \min(k, r) u_k^2 \leq r \norm{u}_2^2 \leq r \:,
\end{align*}
where the last inequality holds since $\norm{u}_2 \leq \norm{u}_p \leq 1$.

Now, we treat the case when $p \in (2, \infty]$.
First, we use the upper bound
\begin{align*}
    A_p \leq \sup_{\norm{u}_p \leq 1} \sum_{k=1}^{T} k u_k^2 := B_p \:.
\end{align*}
We now bound $B_p$ by the duality of $\ell_p$ and $\ell_q$ norms.
We write $\norm{u}_p^p$ as
\begin{align*}
  \norm{u}_p^p = \sum_{k=1}^{T} \abs{u_k}^p = \sum_{k=1}^{T} \abs{\abs{u_k}^2}^{p/2} \:.
\end{align*}
Making a change of variables $w_k \gets \abs{u_k}^2$,
we have by the duality of $\ell_p$ and $\ell_q$ norms that $B_p$ is equivalent to
\begin{align*}
  B_p = \max_{\norm{w}_{p/2} \leq 1} \ip{(1, 2, ..., T)}{w} = \norm{(1, 2, ..., T)}_{\frac{1}{1 - 2/p}} \:.
\end{align*}
Immediately, we can read off that $B_2 = \norm{(1, 2, ..., T)}_\infty = T$
and $B_\infty = \norm{(1, 2, ..., T)}_1 = T(T+1)/2$.
Let us now handle the case when $p \in (2, \infty)$.
\begin{align*}
  B_p &= \norm{(1, 2, ..., T)}_{\frac{1}{1-2/p}} \\
	&= \left( \sum_{k=1}^{T} k^{\frac{p}{p-2}} \right)^{1-2/p} \\
  &\leq \left(  \int_1^{T+1} x^{\frac{p}{p-2}} \; dx \right)^{1-2/p} \\
  &\leq \left( \frac{p-2}{2(p-1)} (T+1)^{2(p-1)/(p-2)}  \right)^{1-2/p} \\
  &= \left( 2^{p/(p-2)} \frac{p-2}{p-1} T^{2(p-1)/(p-2)}  \right)^{1-2/p} \\
  &= \left( \frac{p-2}{p-1} \right)^{(p-2)/p} T^{2(p-1)/p} \:.
\end{align*}
Hence, we have that $B_p \leq \left( \frac{p-2}{p-1} \right)^{(p-2)/p} T^{2(p-1)/p}$ when $p \in (2, \infty)$.
Notice that when $p \searrow 2$ or $p \nearrow \infty$ we have that
this bound approaches $T$ and $T^2$, respectively, which is correct up to constants.
The bound now follows since
$\sup_{p > 2} \left( \frac{p-2}{p-1} \right)^{(p-2)/p} = 1$ and $T = 2r$.
\end{proof}
Finally, we establish a simple lower bound on the
$\Hinf$-norm.
For what follows, let $F \in \C^{r \times r}$
be the un-normalized Discrete Fourier Transform (DFT) matrix
(i.e. $F^{-1} = \frac{1}{r} F^*$).
\begin{prop}
\label{prop:dft_lower_bound}
For any $g \in \C^r$, we have that
\begin{align}
  \sup_{z \in \mathbb{T}} \bigabs{ \sum_{k=0}^{r-1} g_k z^{-k} } \geq \norm{Fg}_\infty \:. \label{eq:hinf_lower_bound}
\end{align}
\end{prop}
\begin{proof}
For any $r$ set of points $z_0, ..., z_{r-1} \in \mathbb{T}$
with $z_m = e^{j\theta_m}$,
\begin{align*}
  \sup_{z \in \mathbb{T}} \bigabs{ \sum_{k=0}^{r-1} g_k z^{-k} } &\geq \max_{m=0, ..., r-1} \bigabs{ \sum_{k=0}^{r-1} g_k z_m^{-k} } \:.
\end{align*}
The result now holds by choosing $\theta_m := 2\pi m/r$.
\end{proof}

We are now in a position to prove
Theorem~\ref{thm:main_result_lower_bound}.
Our approach is based on Fano's inequality combined
with a reduction to multi-way hypothesis testing, which is
one of the standard approaches in statistics for establishing
minimax risk lower bounds (see e.g.~\cite{yu97}).

Recall that $F \in \C^{r \times r}$ is the un-normalized DFT matrix.
Fix a $\delta > 0$,
and define $r$ systems $G_1, ..., G_r$ with parameters $g_k := \delta F^{-1} e_k$,
where $e_k$ is the $k$-th standard basis vector.
Using \eqref{eq:hinf_lower_bound}, for $k \neq \ell$,
\begin{align*}
  \norm{G_k - G_\ell}_\infty \geq \delta \norm{F (F^{-1} e_k - F^{-1} e_\ell)}_\infty = \delta \norm{e_k - e_\ell}_\infty = \delta \:.
\end{align*}
Now let $\Pr_{g_k} := \otimes_{i=1}^{m} N( \Toep_{T \times T}(g_k) u_i, \sigma^2 I_T)$.
Define the random variable $J$ to be uniform on the set $\{1, ..., r\}$, and
the random variable $Z$ with conditional distribution $Z|J{=}k \sim \Pr_{g_k}$.
By Fano's inequality, we have that
\begin{align}
  R := \inf_{\widehat{G}} \sup_{G \in \mathscr{H}_r} \E\hinfnorm{\widehat{G} - G} \geq \frac{\delta}{2} \left[ 1 - \frac{I(Z;J) + \log{2}}{\log{r}} \right] \:,
\end{align}
where $I(Z;J)$ is the mutual information between $Z$ and $J$.
The remainder of the proof involves
choosing a $\delta$ such that
$\frac{I(Z;J) + \log{2}}{\log{r}} \leq 1/2$, from which we conclude $R \geq \delta/4$.

By joint convexity of relative entropy and the fact that $x \mapsto \Toep(x)$ is a linear operator,
letting $k, \ell$ be independent random indices
distributed uniformly on $\{1, ..., r\}$, we have that
\begin{align*}
  I(Z;J) &\leq \E_{k,\ell}[ D_{KL}(\Pr_{g_k}, \Pr_{g_\ell}) ] \\
  &= \frac{\delta^2}{2\sigma^2} \sum_{i=1}^{m} \E_{k,\ell}[ \norm{\Toep_{T \times T}(F^{-1} (e_k-e_\ell))u_i}^2_2 ] \\
  &\leq \frac{2\delta^2}{\sigma^2} \sum_{i=1}^{m} \E_{k}[ \norm{\Toep_{T \times T}(F^{-1} e_k)u_i}^2_2 ] \:,
\end{align*}
where the first equality holds from \eqref{eq:kl_div}
and the last inequality uses $\norm{x+y}^2 \leq 2(\norm{x}^2 + \norm{y}^2)$
which holds for any norm $\norm{\cdot}$ and all vectors $x,y$.
Now we need to upper bound the expectation $\E_k[\norm{\Toep_{T \times T}(F^{-1} e_k) u_i}^2_2]$ for all $i=1, ..., m$.
If we let $f_0^*, ..., f_{r-1}^*$ denote the rows of $F$,
the columns of $F^{-1}$ are $\frac{1}{r} f_0, ..., \frac{1}{r} f_{r-1}$.
Furthermore, each $f_k = \varphi(z_k)$ for $z_k = e^{2\pi j k/r}$.
Using the linearity of $x \mapsto \Toep_{T \times T}(x)$
and the definition of $M_r$ \eqref{eq:Mr} and $A_p$ \eqref{eq:AP},
\begin{align*}
  \E_k[\norm{\Toep_{T \times T}(F^{-1} e_k)u_i}^2_2] &= \frac{1}{r} \sum_{k=1}^{r} u_i^* \Toep_{T \times T}(F^{-1} e_k)^*\Toep_{T \times T}(F^{-1} e_k) u_i \\
  &= \frac{1}{r} u_i^* \left(\sum_{k=1}^{r} \Toep_{T \times T}(F^{-1} e_k)^*\Toep_{T \times T}(F^{-1} e_k) \right) u_i \\
  &= \frac{1}{r} u_i^* \left(\sum_{k=0}^{r-1} \Toep_{T \times T}(r^{-1} f_k)^*\Toep_{T \times T}(r^{-1} f_k) \right) u_i \\
  &= \frac{1}{r} u_i^* \left(\frac{1}{r^2} \sum_{k=0}^{r-1} \Toep_{T \times T}(f_k)^*\Toep_{T \times T}(f_k) \right) u_i \\
  &= \frac{1}{r^2} u_i^* M_r u_i \\
  &\leq \frac{1}{r^2} A_p \:.
\end{align*}
Therefore, the mutual information $I(Z;J)$ is bounded above by
\begin{align*}
  I(Z;J) \leq \frac{2\delta^2 m}{\sigma^2 r^2} A_p \:.
\end{align*}
Setting $\delta = \sqrt{\frac{\sigma^2}{8} \frac{r^2 \log{r}}{A_p m} }$,
and using our assumption that $r \geq 16$, it is straightforward to check that
\begin{align*}
  \frac{I(Z;J) + \log{2}}{\log{r}} \leq \frac{ \frac{2\delta^2 m}{\sigma^2 r^2} A_p  + \log{2} }{\log{r}} \leq \frac{1}{2} \:.
\end{align*}
Hence we conclude that
\begin{align*}
  R \geq \frac{\sigma}{8\sqrt{2}} \sqrt{\frac{r^2 \log{r}}{A_p m}} \:.
\end{align*}
By Proposition~\ref{prop:AP_bound}, when $p \in [1,2]$ we have $A_p = r$, in which case
\begin{align*}
  R \geq \frac{\sigma}{8\sqrt{2}} \sqrt{\frac{r \log{r}}{m}} \:.
\end{align*}
On the other hand, when $p \in (2, \infty]$, we have $A_p \leq 4r^{2(p-1)/p}$, in which case
\begin{align*}
  R \geq \frac{\sigma}{16\sqrt{2}} \sqrt{\frac{r^{2/p} \log{r}}{m}} \:.
\end{align*}
The concludes the proof of Theorem~\ref{thm:main_result_lower_bound}.


\section{Finite Truncation Error Analysis for Stable Systems}
\label{sec:truncation}

In Section~\ref{sec:fir_sysid}, we presented both probabilistic guarantees and
experiment design for identification of FIR systems of length $r$, which were
independent of any system specific properties of $G$.  In this section, we
analyze how system behavior affects the necessary truncation length needed to
reach a desired approximation error tolerance.

In order to provide guarantees, we require that the underlying system $G$ is
stable with stability radius $\rho \in (0, 1)$.  A standard fact states that
stability is equivalent to the existence of a constant $C > 0$ such that the
tail decay on the coefficients of the Laurent expansion $G =
\sum_{k=0}^{\infty} g_k z^{-k}$ satisfies the following condition
\begin{align}
    \abs{g_k} \leq C \rho^k \:, \:\: k \geq 1 \:. \label{eq:assumption_tail_decay}
\end{align}
Under this assumption, a simple calculation reveals that as long as
\begin{align*}
    r \geq \frac{1}{1-\rho} \log\left( \frac{C}{\varepsilon(1-\rho)} \right) \:,
\end{align*}
then we have that the approximation error $\norm{G - G_r}_\infty$ satisfies
$\norm{G - G_r}_\infty \leq \varepsilon$.

Unfortunately, without more knowledge of the system at hand, a bound on $C$ in
\eqref{eq:assumption_tail_decay} is hard to characterize.  However, by slightly
relaxing the decay condition \eqref{eq:assumption_tail_decay}, we are able to
derive a tail bound using system-theoretic ideas.
Intuitively, if a system has long transient behavior, then we expect the
constant $C$ in \eqref{eq:assumption_tail_decay} to be large, since in order to
obtain a small approximation error one needs to capture the transient behavior.
The next lemma shows that the $\Hinf$-norm provides a sufficient
characterization of this transient behavior.
This result is due to Goldenshluger and Zeevi~\cite{goldenshluger01}.
We include the proof for completeness.
\begin{lem}[Lemma 1, Goldenshluger and Zeevi~\cite{goldenshluger01}]
\label{lem:bounding_C}
Let $G(z) = \sum_{k=0}^{\infty} g_k z^{-k}$ be a stable SISO LTI system
with stability radius $\rho \in (0, 1)$.
Fix any $\gamma$ satisfying $\rho < \gamma < 1$.
Then for all $k \geq 1$,
  \begin{align*}
      \abs{g_k} \leq \hinfnorm{G(\gamma z)} \gamma^{k} \:.
  \end{align*}
\end{lem}
\begin{proof}
Define the function $H(z) := G(z^{-1})$, which is analytic for all $\abs{z} \leq 1/\gamma$.
It is easy to check that $k$-th derivative of $H(z)$ evaluated at zero is $H^{(k)}(0) = k! g_k$.
Therefore,
\begin{align*}
  k!\abs{g_k} &= \abs{H^{(k)}(0)} \leq k! \gamma^{k} \max_{\abs{z} \leq 1/\gamma} \abs{H(z)} = k! \gamma^k \max_{\abs{z} \geq \gamma} \abs{G(z)} \\
  &= k!\gamma^k \max_{\abs{z} \geq 1} \abs{G(\gamma z)}
  = k! \gamma^k \hinfnorm{G(\gamma z)} \:.
\end{align*}
Above, the first inequality is Cauchy's estimate formula for analytic functions,
and the last equality follows from the maximum modulus principle.
\end{proof}

We note that the technique used in Lemma~\ref{lem:bounding_C} of considering
the proxy system $G(\gamma z)$ instead of $G(z)$ directly also appears in \cite{boczar17}
in the context of certifying exponential rates of convergence for linear dynamical systems.


\section{Robust Controller Design}
\label{sec:eval}

In Section~\ref{sec:fir_sysid}, we described how to obtain a
FIR system $G_{\mathrm{fir}}$ with a probabilistic guarantee
that $G = G_{\mathrm{fir}} + \Delta$,
where $\Delta$ is an LTI system satisfying $\norm{\Delta}_\infty \leq \varepsilon$.
This description of $G$ naturally lends itself to many robust control
synthesis methods. In this section, we describe the application of
one particular method based on $\mathcal{H}_\infty$ loop-shaping to a particular unknown plant.

Suppose that $G$ is itself an FIR described with $z$-transform
\begin{align}
    G(z) = \abs{w_0} + \sum_{k=1}^{149} \abs{w_k} \rho^{k-1} z^{-k} \:, \:\: \rho=0.95 \:, \label{eq:g_unknown}
\end{align}
where $w_k \sim N(0, 1)$ are independent Gaussians.
In this section, we will detail the design of a reference tracking controller
for $G$ using probabilistic guarantees.

\subsection{Computing bounds}
\label{sec:eval:bounds}

While the non-asymptotic bounds of Section~\ref{sec:fir_sysid} and
Section~\ref{sec:truncation} give us upper bounds on the error of noisy FIR
approximation, the constant factors in the bounds are not optimal.  Hence,
strictly relying on the bounds will cause oversampling by a constant factor of
(say) $10$ or more.  For real systems, this is extremely undesirable--
using the sharpest bound possible is of great practical interest.
Fortunately, we can do this via simple Monte--Carlo simulations, which we detail in
Section~\ref{sec:appendix:monte_carlo} the appendix. For now, we describe the results of these simulations.

Our first Monte--Carlo simulation establishes
that $G$ satisfies the tail decay specified in
\eqref{eq:assumption_tail_decay}
with $C = 3.9703$ and $\rho = 0.95$.
If we truncate $G$ with $r=75$, we see that our worst-case bound
on $\norm{G - G_r}_\infty$ is
$\norm{G -
G_r}_\infty\leq C \frac{\rho^{r-1}}{1-\rho} = 3.9703 \times
\frac{0.95^{74}}{1-0.95} = 1.7840$.  In general, assuming we have no other
information about $G$ other than the bounds on $C$ and $\rho$, this is the
sharpest approximation error bound possible, since for any system with
real-valued, all non-negative Fourier coefficients, the
$\mathcal{H}_\infty$-norm is simply the sum of the coefficients.

However, if we further assume we know the structure of $G$ as in this case
where we know the form of \eqref{eq:g_unknown}, but not the values of $w_k$, we
can further sharper our approximation bound.
Specifically, we know that $E_{\mathrm{approx}} := \norm{G - G_r}_\infty =
\sum_{k=75}^{149} \abs{w_k} \rho^{k-1}$, and hence we can
perform another Monte--Carlo simulation to estimate the tail probability
of this random variable.
The result of our simulation is that $\Pr( E_{\mathrm{approx}} \leq 0.46 ) \geq
0.99$.  This is a substantial improvement over the previous bound of
$E_{\mathrm{approx}} \leq 1.7840$ which only uses the information contained in
the tail decay.

Furthermore, we can use the same trick to sharpest the estimates from
Lemma~\ref{lem:concentration_of_error}.  We perform our final Monte--Carlo
simulation, this time on the random variable $E_{\mathrm{noise}} :=
\frac{\sigma}{\sqrt{N}} \norm{\sum_{k=0}^{74} \xi_k z^{-k}}_\infty$, with
$\sigma^2 = 1$ and $\xi_k \sim N(0, 1)$. Note that this corresponds to choosing
the inputs to the system as impulse responses, which we recall from
Section~\ref{sec:sysid:aopt} is optimal under $\ell_2$-power constraints.
Doing this simulation, we obtain that $\Pr(E_{\mathrm{noise}} \leq 3.5954) \geq
0.99$.

\subsection{Controller design}

\begin{figure}[ht]
\centering
\includegraphics[width=0.7\textwidth]{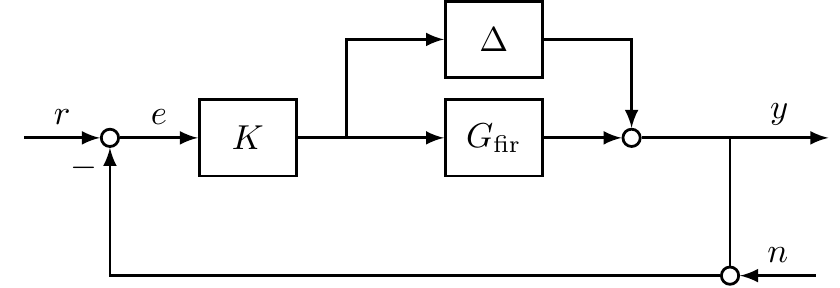}
\caption{Closed-loop experimental setup. The goal is to design the controller $K$.
$G_{\mathrm{fir}}$ is estimated from noisy output data,
and $\norm{\Delta}_\infty$ is bounded via Monte--Carlo simulations.
}
\label{fig:clp}
\end{figure}

Our goal is to design a controller $K$ in the setup described
in Figure~\ref{fig:clp}, under the assumption that
$\norm{\Delta}_\infty \leq E_{\mathrm{approx}} + E_{\mathrm{noise}} \leq 4.0554$.
This assumption comes from the calculations in Section~\ref{sec:eval:bounds}.
We note that $\norm{\Delta}_\infty / \norm{G}_\infty$ fluctuates between
10-20\%, so $G_{\mathrm{fir}}$ is a relatively coarse description of $G$.
We use standard loop-shaping performance goals (see e.g.
\cite{doyle90}).
Let $T_{r \mapsto e}$ and $T_{n \mapsto e}$ denote the
transfer functions from $r \mapsto e$ and $n \mapsto e$, respectively.
At low frequencies, we would like
$\abs{T_{r \mapsto e}}$ to have small gain, and at high frequencies we would like
$\abs{T_{r \mapsto e}} \leq 2$.
Similarly, we would like $\abs{T_{n \mapsto e}} \leq 2$ at low frequencies
and $\abs{T_{n \mapsto e}}$ small at high frequencies.
Of course, we would like these goals to be achieved, in addition to
closed loop stability, for all $G = G_{\mathrm{fir}} + \Delta$.

We proceed in two steps. We first design a controller with the nominal $G_{\mathrm{fir}}$
using $\mathcal{H}_\infty$ loop-shaping synthesis (\texttt{mixsyn} in MATLAB).
We choose weights to encourage our performance goals on $T_{r \mapsto e}$
and $T_{n \mapsto e}$ to be met.
Next, we check that our performance goal is met, in addition to robust
stability.  To make the computation easier, we check the performance goals
separately.  First, it is well known (see e.g.~\cite{doyle90}) that the goal on
$T_{r \mapsto e}$ is met (in addition to robust stability) if the following
holds
\begin{align}
  \norm{ \abs{W_1 S} + \gamma \abs{KS} }_\infty < 1 \:, \label{eq:robust_inequality}
\end{align}
where $S = \frac{1}{1+KG_{\mathrm{fir}}}$ and $\gamma = 4.0554$.
Specifically, under \eqref{eq:robust_inequality},
the closed loop with $K$ in feedback with $G$ is stable and achieves the performance
guarantee $\abs{T_{r \mapsto e}(z)} \leq \frac{1}{\abs{W_1(z)}}$ for every frequency $z \in \Torus$.
On the other hand, to the best of our knowledge no
simple expression for the performance goal on $T_{n \mapsto e}$ exists, so we resort
to a standard structured singular value (SSV) calculation~\cite{packard93}.

We generate our controller $K$ via the following MATLAB commands
\begin{verbatim}
w_c = 0.07;                            % Cross-over freq
W1 = makeweight(5000, w_c, .5, 1);     % Low-freq disturbance rejection
W2 = 1.5*fir_error_bound;              % Robust stability
W3 = makeweight(.5, 3 * w_c, 5000, 1); % High-freq noise insensitivity
P = augw(G_fir, W1, W2, W3);
K = hinfsyn(P);
\end{verbatim}

\begin{figure}[t!]
  \centering
  \begin{minipage}[t]{0.45\textwidth}
  \begin{center}
  \includegraphics[width=\columnwidth]{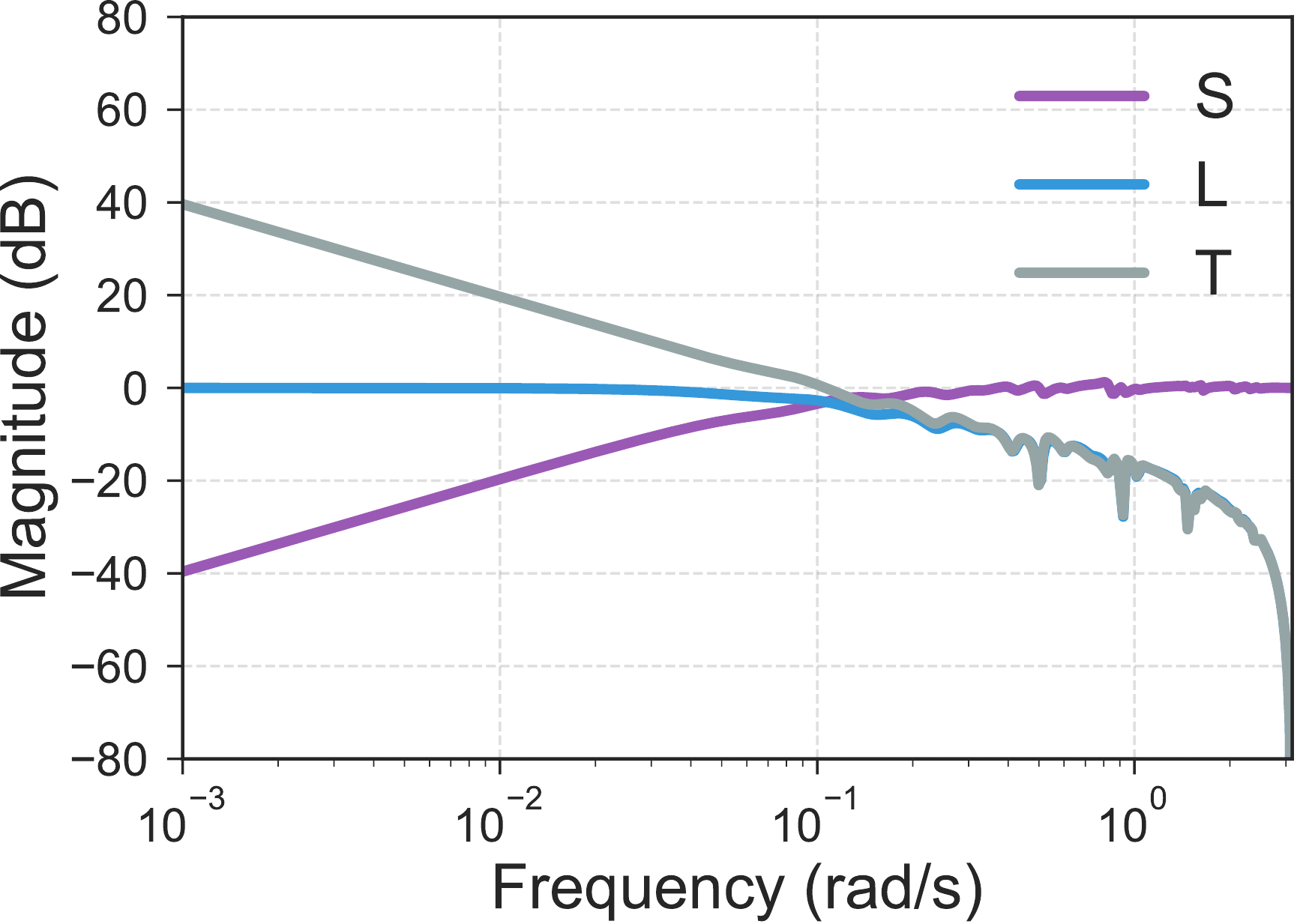}
  \end{center}
  \caption{Loop-shaping curves from the experimental setup
      of Figure~\ref{fig:clp}. $L = G_{\mathrm{fir}} K$ denotes the open-loop gain,
      $S = 1/(1+L)$ is the sensitivity function, and $T = 1 - S$.}
    \label{fig:loopshapes}
  \end{minipage}
  \hspace{.06\textwidth}
    \begin{minipage}[t]{0.45\textwidth}
    \begin{center}
  \includegraphics[width=\columnwidth]{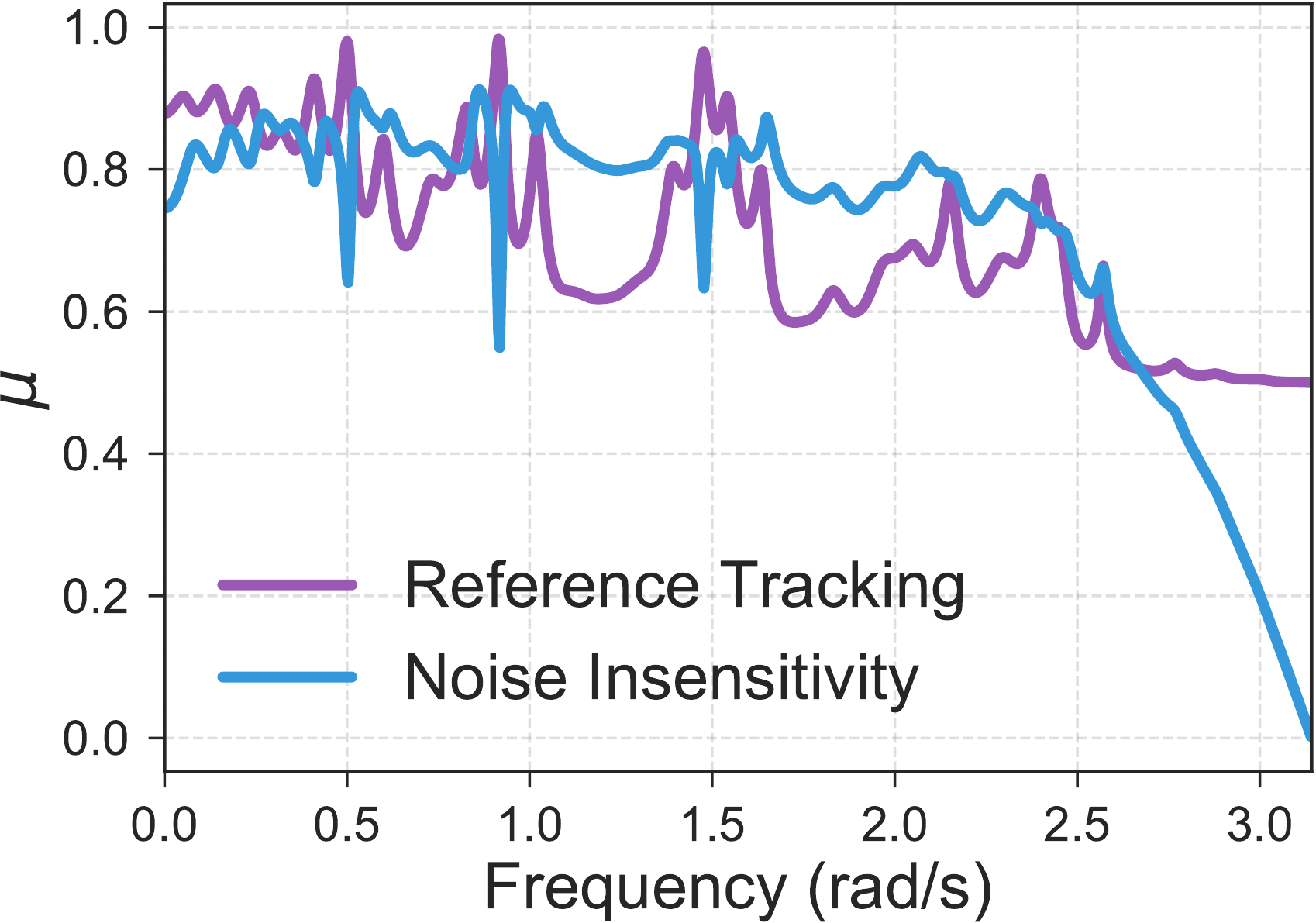}
  \end{center}
\caption{The pointwise frequency $\mu$ value for both
        reference tracking $T_{r\mapsto e}$
        and noise insensitivity $T_{n\mapsto e}$.
        Robust performance is guaranteed as the curve lies below 1
        at all frequencies.
        }
\label{fig:robperf}
  \end{minipage}
  \vspace{0.1in}
\end{figure}

In Figure~\ref{fig:loopshapes}, we plot the open loop gain $L = G_{\mathrm{fir}} K$,
sensitivity function $S = 1/(1+L)$, and complementary sensitivity function $T = 1 - S$.
Here, we see that the cross-over frequency $\omega_c \approx 0.1$.
Next, in Figure~\ref{fig:robperf}, we plot the $\mu$ values for both the
reference tracking objective $T_{r \mapsto e}$ and the noise insensitivity
objective $T_{n \mapsto e}$, and check that both curves lies below 1 for all
frequencies.  Recall that this means that $G$ in feedback with $K$ is not only
exponentially stable, but also satisfies both performance guarantees.
Finally, in Figure~\ref{fig:clpsim}, we plot the output $y$ as a function of a
noisy square wave input $u$, to show the desired reference tracking behavior,
on both the closed loop simulation (with $G_{\mathrm{fir}}$), and the actual
closed loop behavior (with $G$). This shows that, while the model $G_{\mathrm{fir}}$
was a coarse grained description of $G$ with up to 20\% relative error,
it was faithful enough to allow for a robust controller design.

\begin{figure}[t!]
  \centering
  \begin{minipage}[t]{0.45\textwidth}
  \begin{center}
  \includegraphics[width=\columnwidth]{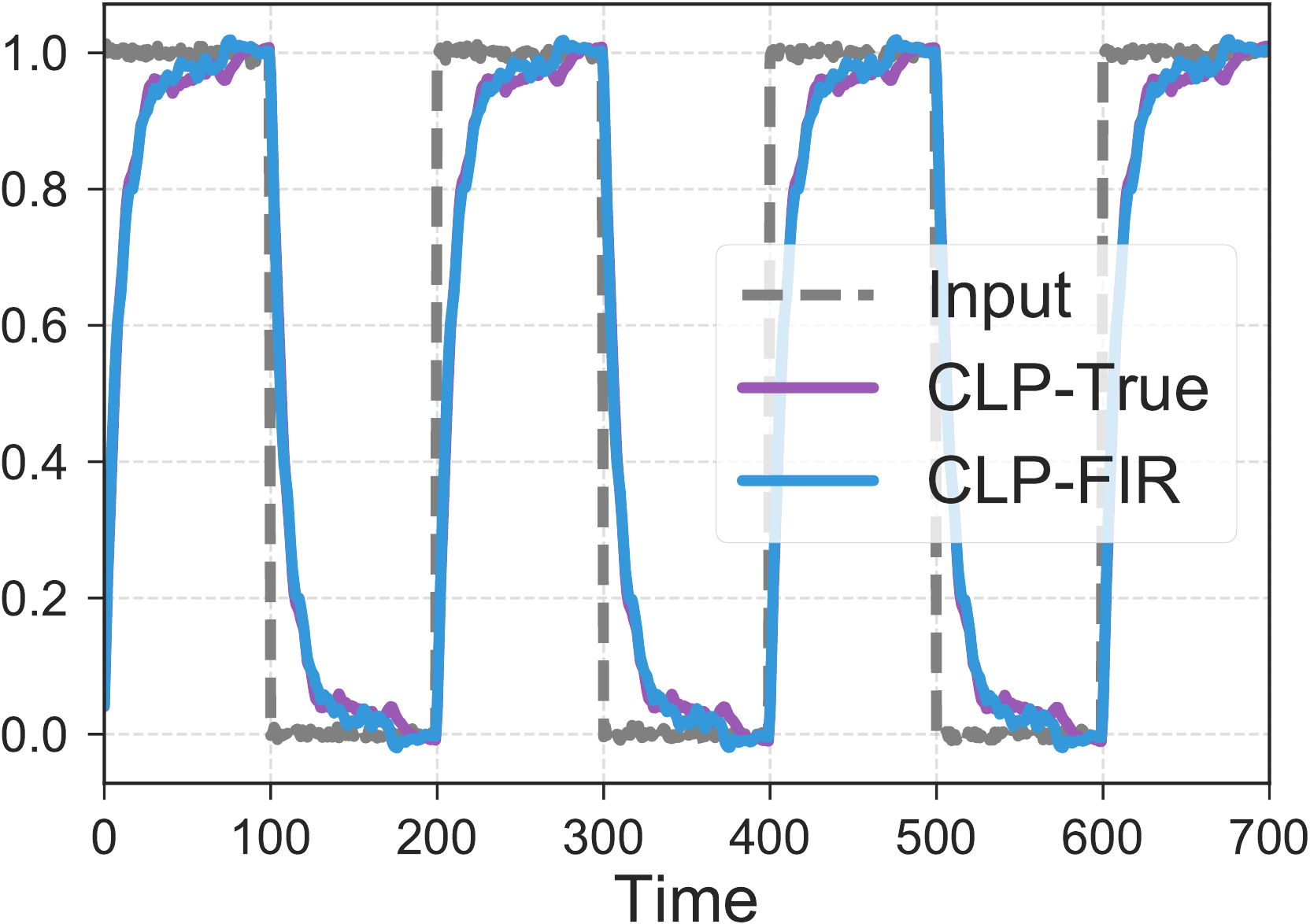}
  \end{center}
\caption{Reference tracking behavior of the closed loop with the model $G_{\mathrm{fir}}$ and the actual plant $G$.}
\label{fig:clpsim}
  \end{minipage}
  \hspace{.06\textwidth}
    \begin{minipage}[t]{0.45\textwidth}
    \begin{center}
  \includegraphics[width=\columnwidth]{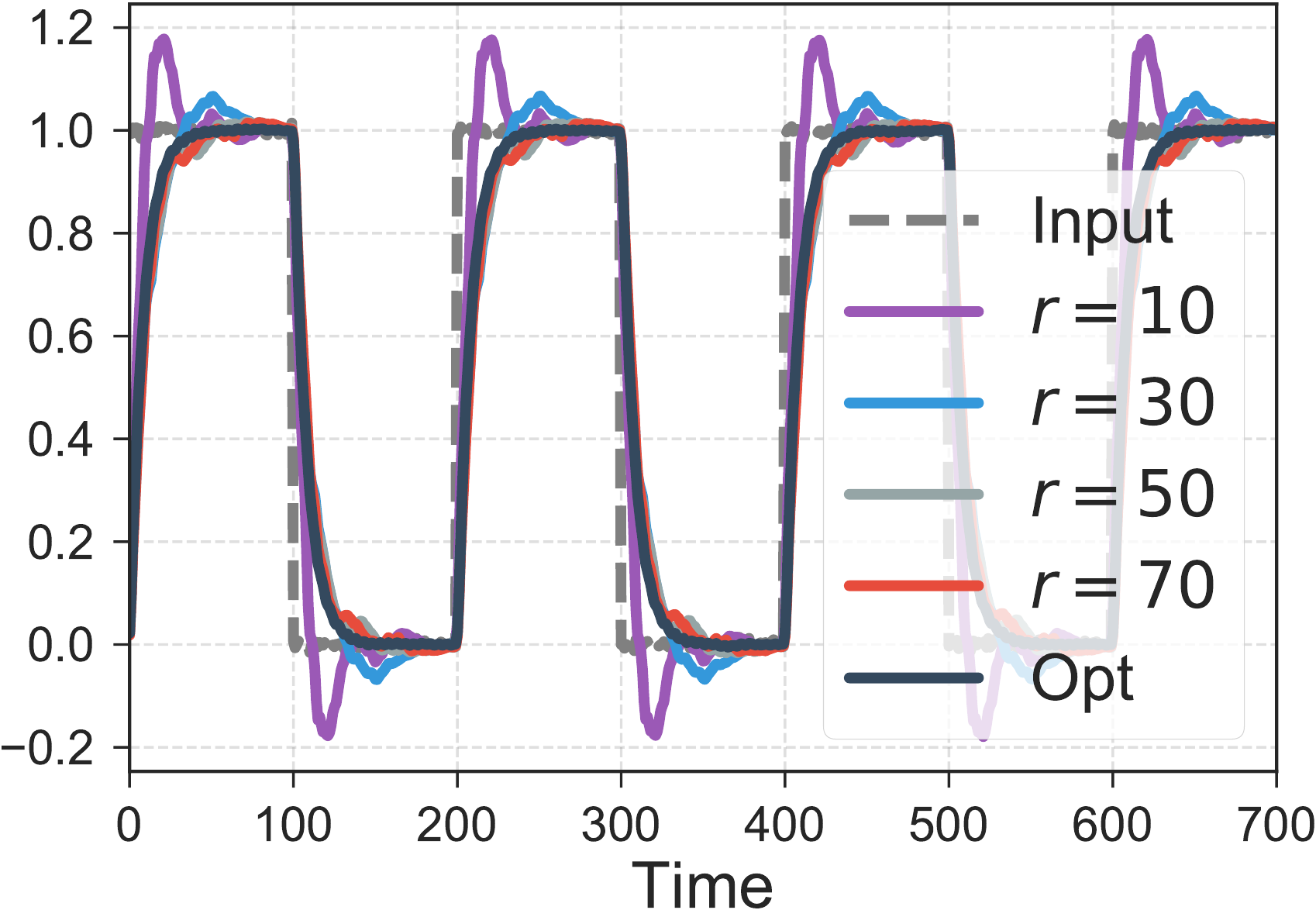}
  \end{center}
\caption{Reference tracking behavior as the FIR truncation
        length is varied from $r=10,30,50,70$.}
\label{fig:vary_r}
  \end{minipage}
  \vspace{0.1in}
\end{figure}

\subsection{Varying truncation length}
We next study the effect of truncation length $r$ on controller design.
In Figure~\ref{fig:vary_r}, we assume the same setup and
performance goals as the previous section,
but vary the truncation length $r \in \{10, 30, 50, 70\}$.
We also include the result of a controller design which has full
knowledge of the true system $G$, which we label as $\mathsf{Opt}$.
We see that for $r=10$, the resulting controller unsurprisingly
has undesirable overshoot behavior. However, as $r$ increases
the resulting controller mimics the behavior of $\mathsf{Opt}$ quite closely.
This plot shows that, at least for reference tracking behavior,
a fairly low-fidelity model suffices.
For instance, across different trials, the relative error of $G_{\mathrm{fir}}$
for $r=30$ fluctuated between 15\% to 30\%, but in many cases $r=30$ was able
to provide reasonable reference tracking behavior.


\section{Conclusion}

This paper explored the use of a coarse-grained FIR model estimated from noisy
output data for control design.
We showed that sharp bounds on the $\mathcal{H}_\infty$ error between the true unknown
plant and the estimated FIR filter can be derived using tools from
concentration of measure, and the constant factors on these bounds can be
further refined via Monte--Carlo simulation techniques.
Finally, we demonstrated empirically that one can perform controller synthesis
using only a coarse-grained approximation of the true system while meeting
certain performance goals.

There are many possible future extensions of our work. We highlight a few ideas
below.

\paragraph{End-to-end Guarantees.}
In Section~\ref{sec:eval}, we studied empirically the performance of the
controller synthesized on our nominal system versus the optimal controller
synthesized on the true system. We would like to push our analysis further
and provide guarantees on this performance gap as a function of our estimation
error. This would allow us to provide a true end-to-end guarantee
on the number of samples needed to control an unknown system to a
certain degree of performance.

\paragraph{MIMO Systems.}
While our approach can be generalized to the MIMO case by estimating filters
for each input/output pair separately, we believe that when the MIMO transfer
matrix has special structure (e.g. low rank), it should be possible to couple
the estimation procedure to reduce the $n^2$ factor increase in sample complexity.
This is motivated by the vast literature on compressed sensing, where 
sparse models embedded in a much larger ambient dimension can be uniquely recovered
with at most a logarithmic factor more samples than the degree of the 
intrinsic sparsity.

\paragraph{Nonlinear Systems.}
An extension of these techniques to nonlinear systems is another
exciting direction. One possible idea is to treat a nonlinear system's
Jacobian linearization as the target unknown system, and fit a FIR using our
techniques by exciting the nonlinear system locally. One would expect that the
controller designed on the FIR would be valid in a neighborhood, and upon
exiting the neighborhood, the process would repeat itself.
The challenge here remains to estimate online the regime for which a controller
is valid.


\section*{Acknowledgements}

We thank 
Orianna DeMasi for helpful comments and suggestions regarding this manuscript,
Kevin Jamieson for insightful discussions regarding Monte--Carlo simulation,
Anders Rantzer for pointing out references related to this work,
Max Simchowitz for guidance regarding the proof of Theorem~\ref{thm:main_result_lower_bound}, and 
Vikas Sindhwani for providing ideas around linear control techniques for nonlinear systems.
RB is supported by the Department of Defense NDSEG Scholarship.
AP gratefully acknowledges generous support from the FANUC Corporation as well
as the National Science Foundation under grant ECCS-1405413.
BR is generously supported by NSF award CCF-1359814, ONR awards
N00014-14-1-0024 and N00014-17-1-2191, the DARPA Fundamental Limits of Learning
(Fun LoL) Program, a Sloan Research Fellowship, and a Google Faculty Award.

{
\if\MODE1\else\small\fi
\bibliographystyle{abbrv}
\bibliography{sector_bib}
}
\clearpage

\appendix

\section{Proof of Lemma~\ref{lem:orthogonal_vectors_M}}
\label{proof:lem_orthogonal_vectors_M}
We will induct on $n$, for which the base cases $n=1$ holds. Denote the $i$-right shift operator applied to $u_k$ by $u_k^{[i]}$. Now, assume the property holds for $n$. This implies that
\begin{align*}
M_{ij} = &\; \sum_{k=0}^{2^n-1} u_k^{[i]} \cdot u_k^{[j]} = 0 \\
M_{ii} = &\; \sum_{k=0}^{2^n-1} u_k^{[i]} \cdot u_k^{[i]} = 2^n(2^n-i)\:.
\end{align*}
Now, construct $\{\tilde u_k\}$ as before, and note that for $l=0,\ldots,2^{n}-1$,
\begin{equation*}
\tilde u_{2l}^{[i]} = \; \begin{cases}
\begin{bmatrix}0_{i\times1}\\u_l\\u_l^{[i]}\end{bmatrix}, & i\leq 2^n \\
\begin{bmatrix}0_{i\times1}\\u_l^{[i-2^n]}\end{bmatrix}, & i >  2^n
\end{cases}\quad\quad
\tilde u_{2l+1}^{[i]} = \; \begin{cases}
\begin{bmatrix}0_{i\times1}\\u_l\\-u_l^{[i]}\end{bmatrix}, & i\leq 2^n \\
\begin{bmatrix}0_{i\times1}\\-u_l^{[i-2^n]}\end{bmatrix}, & i >  2^n
\end{cases}\:.
\end{equation*}
Thus, for $i\leq 2^n$, when $i \neq j$,
\begin{align*}
\tilde M_{ij} = &\; \sum_{k=0}^{2^{n+1}-1} \tilde u_k^{[i]} \cdot \tilde u_k^{[j]} \\
= &\;\sum_{l=0}^{2^n-1} \tilde u_{2l}^{[i]} \cdot \tilde u_{2l}^{[j]}+\tilde u_{2l+1}^{[i]} \cdot \tilde u_{2l+1}^{[j]}\\
= &\;\sum_{l=0}^{2^n-1}  u_l \cdot  u_l+ u_l^{[i]} \cdot  u_l^{[j]} -
 u_l \cdot  u_l+ u_l^{[i]} \cdot  u_l^{[j]} = 0\:.
 \end{align*}
Furthermore,
 \begin{align*}
\tilde M_{ii} = &\;\sum_{k=0}^{2^{n+1}-1} \tilde u_k^{[i]} \cdot \tilde u_k^{[i]} \\
= &\; \sum_{l=0}^{2^n-1} \tilde u_{2l-1}^{[i]} \cdot \tilde u_{2l-1}^{[i]}+\tilde u_{2l}^{[i]} \cdot \tilde u_{2l}^{[i]}\\
= &\; \sum_{l=0}^{2^n-1}  u_l \cdot  u_l+ u_l^{[i]} \cdot  u_l^{[i]} +
 u_l \cdot  u_l+ u_l^{[i]} \cdot  u_l^{[i]} \\
= &\; 2(2^n(2^n-i)) + 2(2^n(2^n))\\
= &\; 2^{n+1}(2^{n+1}-i)\:.
\end{align*}
A similar calculation holds for $i>2^n$. Thus, by induction, the property holds for all $r=2^n$.

\section{Details for Monte--Carlo simulations}
\label{sec:appendix:monte_carlo}

In all of our simulations we are faced with the following problem which we describe
in some generality.
Let $X$ be a random variable distributed according to the law $\Pr$.
We assume we have access to iid samples from $\Pr$. Our goal is to estimate an upper bound on
$\Pr(X \geq t)$ for a fixed $t \in \R$.

If the law $\Pr$ admits a density $f(\cdot)$ with respect to the Lebesgue measure, a
possible solution could be to solve this problem exactly by numerically integrating
\begin{align*}
    \int_{X(\xi) \geq t} X(\xi) f(\xi) \; d\xi \:.
\end{align*}
However, numerical integration does not scale favorably with dimension.
For our experiments, $\xi$ is 75-dimensional, which is
prohibitive for numerical integration.

An alternative approach to numerical integration is to rely on concentration of
measure.  Let $X_1, ..., X_N$ be iid copies of $X$, and let $\Pr^N$ denote the
product measure $\Pr^N = \otimes_{k=1}^{N} \Pr$.
Using a Chernoff bound and defining $F_t := \Pr(X \geq t)$, we have
\begin{align}
    \Pr^N\left( \frac{1}{N} \sum_{k=1}^{N} \ind_{\{X_k \geq t\}} \leq F_t - \varepsilon \right) \leq e^{-N \cdot D(F_t - \varepsilon, F_t)} \:, \label{eq:chernoff_bound}
\end{align}
where $D(p, q) = p \log(p/q) + (1-p) \log((1-p)/(1-q))$ is the KL-divergence between two Bernoulli distributions.
Given a $\delta \in (0, 1)$,
define the random variable $Q$ as the solution to the implicit equation
\begin{align}
    N \cdot D\left(\frac{1}{N} \sum_{k=1}^{N} \ind_{\{{X_k} \geq t\}}, Q\right) = \log(1/\delta) \:. \label{eq:KL_invert}
\end{align}
Note that, from a realization of $X_1, ..., X_N$, the realization of $Q$
from \eqref{eq:KL_invert} can be solved for by numerical root finding.
Plugging the definition of $Q$ back into the Chernoff inequality \eqref{eq:chernoff_bound}, we conclude that
there exists an event $\mathcal{E}$ (in the product $\sigma$-algebra) such that on $\mathcal{E}$ the inequality $F_t \leq Q$ holds,
and furthermore $\Pr^N(\mathcal{E}) \geq 1-\delta$.
This is the methodology which we use to generate all our bounds, with $\delta = 10^{-4}$.
Hence, the statements of the form ``$F_t \leq \gamma$'' in Section~\ref{sec:eval:bounds} should
be understood as operating under the assumption that our implementation of the
simulation chose a particular realization which is contained in the
simulator event $\mathcal{E}$ described previously.

\end{document}